\documentclass[10pt, twoside, a4paper]{article}
\usepackage[utf8]{inputenc}
\usepackage[T1]{fontenc}
\usepackage{mathtools}
\usepackage{amsmath, amssymb, amsthm, bigints} 
\usepackage{amstext, amsfonts, a4}
\usepackage[english]{babel}
\usepackage{stmaryrd}
\usepackage{bbm}
\usepackage{graphicx}
\usepackage[margin=3cm]{geometry}

\theoremstyle{plain}
\newtheorem{Theo}{Theorem}

\newtheorem{Prop}{Proposition} 

\newtheorem{Lemma}{Lemma}

\theoremstyle{definition}
\newtheorem*{merci}{Acknowledgments}

\theoremstyle{remark}
\newtheorem{Case}{Case}

\DeclareMathOperator{\infess}{inf ess}
\DeclareMathOperator{\meas}{meas}
\DeclareMathOperator{\supp}{supp}

\DeclareMathOperator{\dist}{dist}

\DeclareMathOperator{\sign}{sign}

      \def\NN{{\mathbb N}}
   \def\RR{{\mathbb
R}}

 \def\cG{{\mathcal G}} 
 \def\cB{{\mathcal B}} 
  \def\cC{{\mathcal C}}
\def\cI{{\mathcal I}} \def\cO{{\mathcal O}} 
 \def\cJ{{\mathcal J}} 
 \def\cE{{\mathcal E}}

  \def\mfS{{\mathfrak S}}

\begin{document}

\title{Co-rotating vortices with $N$ fold symmetry for the inviscid surface quasi-geostrophic equation}
\author{
\renewcommand{\thefootnote}{\arabic{footnote}}
Ludovic Godard-Cadillac\footnotemark[1], Philippe Gravejat\footnotemark[2]~ and Didier Smets\footnotemark[3]}
\footnotetext[1]{Sorbonne Universit\'e, Laboratoire Jacques-Louis Lions, 4, Place Jussieu, 75005 Paris, France.\\ E-mail: {\tt ludovic.godardcadillac@ljll.math.upmc.fr}}
\footnotetext[2]{CY Cergy Paris Universit\'e, Laboratoire de Math\'ematiques AGM, F-95302 Cergy-Pontoise Cedex, France.\\ E-mail: {\tt philippe.gravejat@cyu.fr}}
\footnotetext[3]{Sorbonne Universit\'e, Laboratoire Jacques-Louis Lions, 4, Place Jussieu, 75005 Paris, France.\\ E-mail: {\tt didier.smets@upmc.fr}}
\date{\today}
\maketitle

\begin{abstract}
We provide a variational construction of special solutions to the generalized surface 
quasi-geostrophic equations. These solutions take the form of $N$ vortex patches with $N$-fold 
symmetry, which are steady in a uniformly rotating frame. Moreover, we investigate their 
asymptotic properties when the size of the corresponding patches vanishes. In this limit, we 
prove these solutions to be a desingularization of $N$ Dirac masses with the same intensity, located 
on the $N$ vertices of a regular polygon rotating at a constant angular 
velocity.
\end{abstract}

\section{Introduction}

In this paper, we consider the generalized surface quasi-geostrophic equation 
\begin{equation}
\label{gSQG}
\tag{gSQG}
\begin{cases} 
\partial_t \vartheta + v \cdot \nabla \vartheta = 0,\\
v = \nabla^\perp \Delta^{-s} \vartheta,
\end{cases}
\end{equation}
for $\vartheta : \RR^2 \times \RR \to \RR$, where $0 < s < 1$ is a fixed parameter. We
use here the notation $\Delta^{-s}$ as a shortcut for $-(-\Delta)^{-s}$. 

The surface quasi-geostrophic equation corresponds to the case $s = \frac{1}{2}$, it is used as a 
model in the context of geophysical fluid mechanics with $\vartheta$ being the potential temperature 
in a rapidly rotating stratified fluid with uniform potential vorticity and subject to Brunt-V\"ais\"al\"a oscillations (see~\cite{HePiGaS1} and the references therein). 

Formally at least, the family of nonlinear transport equations~\eqref{gSQG} interpolates between 
the stationary equation $\partial_t \vartheta = 0$ when $s \to 0$, and the two-dimensional Euler 
equation for incompressible fluids in the limit $s \to 1.$ In that last case, the vector field $v$ is 
the flow velocity and the function $\vartheta$ is the flow vorticity, usually written as $\omega$ in
that context.

From a mathematical point of view, the surface quasi-geostrophic equation has attracted 
interest in particular because, while being two dimensional and therefore simpler to analyze, 
it possesses some scaling analogy with the three-dimensional Euler equation 
(see~\cite{ConMaTa1} and the references therein). 
Local well-posedness of classical solutions was established in~\cite{ConMaTa1},
but solutions with arbitrary Sobolev growth were constructed in~\cite{KiseNaz1} (in a periodic
setting). To our knowledge, establishing global well-posedness of classical solutions, or
alternatively describing their finite time singularities, remains an open problem. Note also that
the global existence of weak solutions in $L^2(\RR^2)$ was shown in~\cite{Resnick0}, but their
non-uniqueness below a certain regularity threshold was highlighted in~\cite{BucShVi1}.

As for the two-dimensional Euler equation, all the radially symmetric 
functions $\vartheta$ are stationary solutions to~\eqref{gSQG}. Exhibiting other global smooth 
solutions seems a challenging issue. A first example was recently provided 
in~\cite{CasCoGS1} by developing a bifurcation argument from a specific
radially symmetric function. The variational construction of an alternative example in the form of a
smooth traveling-wave solution was completed in~\cite{GravejatSmets, GodaCad0}. These latter results
are the starting point of the following analysis.

It is well-known that the two-dimensional Euler equation exhibits point vortex solutions
(see~\cite{MarcPul0} and the references therein), which take the form 
$$\vartheta(x, t) = \sum_{n = 1}^N a_n \delta_{x = x_n(t)}.$$ 
In this formula, the notation $\vartheta$ stands for the vorticity of
the ideal fluid under consideration, and $a_n$ and $x_n(t)$ denote the intensity and the position at
time $t$ of the $n$-th vortex. When the number $N$ of vortices is equal to $2$, two special
configurations are of interest. The first one corresponds to the case $a_1 + a_2 = 0$ for which the
two vortices translate at a constant common velocity. When $a_1 - a_2 = 0$ instead, the two vortices 
rotate around each other at a constant angular velocity, while their midpoint remains fixed. This 
latter configuration generalizes to an arbitrary choice of the number $N \geq 3$, for which $N$ 
vortices, with the same intensity and located on the $N$ vertices of a regular polygon, co-rotate 
at a constant angular velocity.

Point vortices are singular weak solutions to the Euler equation. On finite time intervals, they can 
be approximated through regularization by classical solutions (see e.g.~\cite{MarcPul0}). An
interesting issue is to ask for the existence of regularized solutions which approximate the point
vortices for all time. This desingularization issue was answered positively for numerous vortex
configurations of the Euler equation (see e.g.~\cite{SmetVSc1} and the references therein). In
particular, the construction of smooth traveling-wave solutions corresponding to a vortex pair in
translation was achieved in~\cite{Norbury} (see also~\cite{BergerFraenkel, Burton1}). In parallel,
less regular approximations called vortex patches were constructed both for the vortex pair in
translation~\cite{Turking1, Turking2} and for the ones in rotation~\cite{Turking3}.

The notion of point vortex solutions easily extends to the context of the generalized surface
quasi-geostrophic equations. This was the starting point of the constructions
in~\cite{GravejatSmets, GodaCad0}, which exhibited a smooth traveling-wave solution to~\eqref{gSQG}
corresponding to a vortex pair in translation. In this direction, a natural question is to ask for a
similar construction in the case of a vortex pair in rotation. For $1/2 < s < 1$, this question was
answered positively in~\cite{HmidMat1}, where a pair of vortex patches was shown to exist by a
perturbative argument. In~\cite{HassHmid}, solutions with $N$-fold symmetry 
were constructed for $1/2 < s < 1$ (and also for the Euler equation). They are perturbations of one 
radially symmetric vortex patch and therefore are of different nature 
than desingularized point vortex solutions. While completing this work, smooth desingularized 
point vortex solutions of the form of translating vortex pairs or co-rotating vortices were also 
constructed using a perturbative method in \cite{AoDavDPMW}. 

Our main goal in this paper, in the spirit of~\cite{Turking3} for the Euler equation, 
is to provide a variational construction of families, for arbitrary $0 < s < 1$, of $N$ 
co-rotating vortex patches with $N$-fold symmetry. Additionally, we prove 
these patches to be a desingularization of $N$ point vortices in rotation.

\bigskip

Before stating our main result, we first explicit the equation satisfied by a solution
to~\eqref{gSQG}, which remains steady in a uniformly rotating frame. We introduce an angular
velocity $\alpha$ and we look for solutions under the form
\begin{equation}
\label{def:omega}
\vartheta(x, t) = R_{\alpha t}\omega(x) := \omega(R_{-\alpha t} x),
\end{equation}
where $R_\phi$ stands for the counterclockwise rotation of angle $\phi$. Here, we make the choice to use the
notation $\omega$ for the profile of the solution in the rotating frame in reference to the
vorticity formulation of the two-dimensional Euler equation. In view of the
equation for the velocity $v$ in~\eqref{gSQG}, we next set
\begin{equation}
\label{def:u}
u = - \nabla^\perp \big( k_s \star \omega \big).
\end{equation}
In this identity, the function $k_s$ is the fundamental solution of the operator $(- \Delta)^s$ on
$\RR^2$. This solution is explicitly given by the formula
\begin{equation}
\label{def:k_s c_s}
k_s(x) = \frac{c_s}{|x|^{2(1 - s)}},\quad \text{ with } c_s := \frac{\Gamma(1 - s)}{2^{2s} \pi
\Gamma(s)},
\end{equation}
it is often called the Riesz potential of order $2s$, and denoted also as $I_{2s}.$\\ 
Combining~\eqref{def:omega} and~\eqref{def:u} yields
$$v(x, t) = R_{\alpha t} u(R_{-\alpha t} x),$$
so that the first equation in~\eqref{gSQG} reduces to
$$\alpha x^\perp \cdot \nabla \omega(x) + \nabla^\perp \big( k_s \star \omega \big)(x) \cdot \nabla \omega(x) = 0.$$
Hence, a solution to~\eqref{gSQG}, which is steady in a uniformly rotating frame, satisfies the stationary equation
\begin{equation}
\label{eq:rotating frame dynamics}
\nabla \omega(x) \cdot \nabla^\perp \Big( k_s \star \omega(x) +\frac{\alpha}{2} |x|^2\Big) = 0.
\end{equation}
A possible weak formulation for~\eqref{eq:rotating frame dynamics} is that $\omega$ satisfies 
\begin{equation}
\label{eq:weak formulation}
\int_{\RR^2} \omega(x) \nabla \varphi(x) \cdot \nabla^\perp \Big( k_s \star \omega(x) + 
\frac{\alpha}{2}|x|^2\Big) \, dx = 0,
\end{equation}
for any function $\varphi \in \cC_c^\infty(\RR^2)$. The vortex patch solutions which we will
construct own a priori no more regularity than being in the Lebesgue spaces $L^p(\RR^2)$ for any $1 \leq p
\leq \infty$. Since they are also of compact support, due to standard elliptic theory for 
Riesz potentials the corresponding functions 
$k_s \star \omega$ belong to the Sobolev spaces $\dot W^{2 s, p}(\RR^2)$ for any $1 < p < \infty$. 
For $s < 1/2$, this regularity is not yet sufficient to provide a rigorous meaning to the weak 
formulation in~\eqref{eq:weak formulation}.
For that reason, we will rely on a different and restricted notion of weak solutions. More precisely,
in view of~\eqref{eq:rotating frame dynamics} we shall say that if $\omega \in L^1(\RR^2)$ has
compact support and satisfies
$$
\omega = f(k_s \star \omega + \frac{\alpha}{2}|x|^2),
$$
for some Borel measurable function $f:\,\RR\to\RR$, then it is a weak solution of 
\eqref{eq:rotating frame dynamics}. We will be particularly interested in the case where $f$ is
a step function, and in such cases we will show that these weak solutions are also weak solutions
in the sense of~\eqref{eq:weak formulation} when $s \geq \frac{1}{2}.$

\medskip

A vortex patch $\omega$ writes as
$$\omega = \lambda \mathbbm{1}_\Omega.$$
In that identity, the notation $\mathbbm{1}_\Omega$ refers to the characteristic function of the support
$\Omega$ of the patch. The positive number $\lambda$ is its intensity. Since we aim at desingularizing
$N$ point vortices with equal intensity, it is natural to impose that the $N$ patches own the same
intensity $\lambda$. Since the point vortices are located on the $N$ vertices of a regular polygon,
we also impose the symmetry condition
\begin{equation}
\label{eq:symmetry property}
\omega = R_{\frac{2\pi}{N}}\omega.
\end{equation}
In the sequel, this symmetry property is used in order to restrict the
construction to only one vortex patch inside the angular sector
\begin{equation}
\label{def:mfS_N}
\mfS_N := \Big\{ \big( r \cos(\theta), r \sin(\theta) \big) \in \RR^2 : -
\frac{\pi}{N} < \theta < \frac{\pi}{N} \Big\}.
\end{equation}
Moreover, we also impose the vortex patch to be angular Steiner symmetric with
respect to the angle
$\theta = 0$. Recall that this condition is defined as
\begin{equation}
\label{def:angular-Steiner}
\omega(r, \theta) = \omega^\sharp(r, \theta),
\end{equation}
where the notation $(r, \theta)$, with $r > 0$ and $- \pi/N < \theta < \pi/N$, stands for the usual
polar coordinates. In this definition, the angular Steiner symmetrization $\omega^\sharp$ of the
function $\omega$ for the variable $\theta$ is the unique even function such that
$$
\omega^\sharp(r, \theta) > \nu \quad \text{if and only if} 
\quad |\theta| < \frac{1}{2} \meas 
\Big\{ \theta' \in \Big( - \frac{\pi}{2}, \frac{\pi}{2} \Big) : \omega(r, \theta') > \nu \Big\},
$$
for any positive numbers $r$ and $\nu$, and any $- \pi/N < \theta < \pi/N$. It
turns out that the symmetry condition~\eqref{def:angular-Steiner} provides a
number of simplifications in our variational construction of $N$ vortex patches
(see Appendix~\ref{sec:angular-Steiner} for some useful properties of the
angular Steiner symmetrization). In the sequel, we denote by
$L_{\text{sym}}^p(\RR^2)$ the subset of functions in $L^p(\RR^2)$, which
satisfy both~\eqref{eq:symmetry property} and~\eqref{def:angular-Steiner}. 

We are now in position to state our main result.

\begin{Theo}
\label{thm:patch-exist}
Let $0<s<1$ be given, and let $\lambda>0$ be sufficiently large, depending only on $s.$ 
There exist an angular velocity $\alpha_\lambda$ and a function 
$\omega_\lambda \in L_{\text{sym}}^\infty(\RR^2)$, which satisfy the following properties.

\noindent $(i)$ The vorticity $\omega_\lambda$ is a co-rotating vortex patch solution to
\eqref{gSQG} with intensity $\lambda$, in the sense that there exists 
$\mu_\lambda \in \RR$ such that
\begin{equation}
\label{eq:patch}
\omega_\lambda = \lambda \mathbbm{1}_{\{ \psi_\lambda > 0 \}},
\end{equation}
with
\begin{equation}
\label{def:psi-lambda}
\psi_\lambda(x) := \frac{\Gamma(1 - s)}{2^{2s} \pi \Gamma(s)} \int_{\RR^2} 
\frac{\omega_\lambda(x')}{|x - x'|^{2(1 - s)}} \, dx' 
+ \frac{\alpha_\lambda}{2} |x|^2 - \mu_\lambda.
\end{equation}

\noindent $(ii)$ If $1/2 \leq s < 1$, then $\omega_\lambda$ is also a weak solution
to~\eqref{eq:rotating frame dynamics} in the sense of~\eqref{eq:weak
formulation}, where $\alpha \equiv \alpha_\lambda.$

\noindent $(iii)$ Let $\Omega_\lambda := \{ \psi_\lambda > 0 \}$ be the support of $\omega_\lambda$.
There exists a positive number $C$, independent of $\lambda$, such that
$$
	\Omega_\lambda \cap \mfS_N \subseteq B \Big( (1,0), \frac{C}{\sqrt{\lambda}} \Big).
$$
In particular, the support $\Omega_\lambda$ has at least $N$ connected
components when $\lambda$ is sufficiently large.

\noindent $(iv)$ In the limit $\lambda \to \infty$,
\begin{equation}
\label{eq:lim_omega}
\omega_\lambda \to \sum_{n = 0}^{N - 1}\delta_{R_\frac{2 n \pi}{N}(1, 0)},
\end{equation}
weakly in the sense of measures, and if $\frac{1}{2} \leq s < 1$ then the angular speeds satisfy
\begin{equation}
\label{eq:lim_alpha}
\alpha_\lambda \to 
\sum_{n = 1}^{N - 1} \frac{c_s (1 - s)}{\big| (1, 0) - R_\frac{2 n \pi}{N}(1, 0) \big|^{2 (1 - s)}}.
\end{equation}

\noindent $(v)$ Concerning $\mu_\lambda$, we have the asymptotic bounds 
\begin{equation}
\label{eq:mu behavior}
0 < \liminf_{\lambda \to +\infty} \frac{\mu_\lambda}{\lambda^{1 - s}} \leq 
\limsup_{\lambda \to +\infty} \frac{\mu_\lambda}{\lambda^{1 - s}} <
+\infty.
\end{equation}
\end{Theo}



In common terms, $\omega_\lambda$ is therefore a desingularized solution 
of the~\eqref{gSQG} equation with $N$ patches rotating at a constant 
angular velocity $\alpha_\lambda.$ 

Note that taking the limit $s \to 1^-$ of the expression on the right-hand-side 
of~\eqref{eq:lim_alpha} in the particular case $N=2$, yields the constant $\frac{1}{4\pi}.$ That 
does correspond exactly to the speed of rotation of a pair of
point vortices on the unit circle evolving according to the Euler 
equations.

We mention that in the case of the Euler equations the
divergence speed of the parameter $\mu_\lambda$ with respect 
to $\lambda$ is logarithmic (see~\cite{Turking3}). In that case
the full limit is known to exist for the quotient in~\eqref{eq:mu behavior}.

We also do not know whether the support of each patch of $\omega_\lambda$
is convex or connected, even for $\lambda$ large. For values of 
$s$ smaller than one half this seems particularly challenging, since the
weak formulation~\eqref{eq:weak formulation} a priori makes no sense.

\bigskip

In the next section we present the outline of the proof of Theorem
\ref{thm:patch-exist}. More precisely, we will describe the variational
strategy which is followed, involving penalized regularized problems, and
describe the various uniform and localization estimates which eventually allow
us to pass to the limit and complete the proof of Theorem
\ref{thm:patch-exist}. The intermediate steps are all stated within that
section, but the proofs are postponed to Section 3. Finally,
Appendix A contains a number of properties relating the Steiner symmetrization
and functionals which appear at some point in our analysis.

\section{Outline of the proof of Theorem~\ref{thm:patch-exist}}
\label{sec:strategy}

In this section we describe in detail our strategy to prove Theorem~\ref{thm:patch-exist}, 
and we state all the intermediate results. The proofs of the later are postponed
to Section~\ref{sec:proofs}.

\subsection{Construction of the co-rotating vortex patches}

The Hamiltonian nature of the Euler equations (see e.g. Arnold~\cite{Arnold}) 
has long been a fruitful source for the construction of steady solutions. We
follow a similar strategy in the context of the~\eqref{gSQG} equation, and for that 
purpose we consider the energy of the fluid
\begin{equation}
\label{def:Es}
E_s(\omega) := \frac{1}{2} \int_{\RR^2} \int_{\RR^2} 
k_s(x - x') \omega(x) \omega(x') \, dx \, dx',
\end{equation}
and its impulse
\begin{equation}
\label{def:L}
L(\omega) := \int_{\RR^2} |x|^2 \omega(x) \, dx.
\end{equation}
Both quantities are conserved for smooth solutions of equation~\eqref{gSQG}, 
and at least formally equation~\eqref{eq:rotating frame dynamics} arises as 
the Euler-Lagrange equation associated to these two functionals for some
Lagrange multiplier $\alpha.$ 

Similar to the Euler equation, equation~\eqref{gSQG} is also a transport 
equation. In particular, any functional or constraint 
which is invariant by rearrangement is formally 
preserved by the flow. This applies to the total circulation of the fluid
\begin{equation}
\label{def:C}
M(\omega) := \int_{\RR^2} \omega(x) \, dx,
\end{equation}
and also to upper and lower bounds on the scalar function $\omega$, 
constraints that we use in the sequel.

\subsubsection{A constrained maximization problem with compact setting}

There are a number of obstacles to build a meaningful and effective 
variational problem associated to the functionals $E_s$, $L$ and $M.$ 

First, the energy functional $E_s$ is unbounded under $L^1$-type constraints 
such as those given by $L$ or $M$. To overcome this, for $\lambda > 0$ 
we consider the additional $L^\infty$-type contraints set 
$$
L_\lambda^\infty(\RR^2) := \Big\{ \omega \in L^\infty(\RR^2) : 0 
\leq \omega \leq \lambda \text{ a.e.} \Big\}.
$$

Next, the physical domain is infinite and the functional spaces on which $E_s$, 
$L$ and $M$ are naturally defined are different. Following Turkington 
\cite{Turking3} in his study of vortex patches for the Euler equation, we 
consider the section of annulus
$$
S := \Big\{ (r \cos(\theta), r \sin(\theta)) : \frac{1}{2} \leq r \leq 2 
\text{ and } - \frac{\pi}{2N} \leq \theta \leq\frac{\pi}{2N} \Big\},
$$
and further restrict our study to 
the function set
\begin{equation}
\label{def:Xlambda}
X_\lambda := \Big\{ \omega \in L^\infty_\lambda(\RR^2) \text{ s.t. } 
\omega \circ R_\frac{2\pi}{N} = \omega \text{ and } \omega = 0 \text{ a.e. on } 
\mfS_N \setminus S \Big\}.
\end{equation}
The symmetry assumption is natural since it reflects the ones of the
functionals, but the imposed vanishing of $\omega$ in $\mfS_N \setminus S$ 
is evidently not a natural constraint, it makes the functional setting 
easier and provides some form of compactness to the problem, but eventually we
will need to show that it is
not activated. The introduction of $S$ is motivated by the fact that, 
in the limit $\lambda \to \infty$ and in view of the symmetry assumption, 
the vortex patches which we will construct are expected to be supported 
into $N$ small balls centered at the $N$ vertices 
$x_n = (\cos(2 n \pi/N), \sin(2 n \pi/N))$ of a regular polygon.

Observe that for functions $\omega \in X_\lambda$ the energy $E_s$ may be rewritten as
\begin{equation}
\label{eq:polar-Es}
E_s(\omega) = \frac{N}{2} \int_S \int_S \kappa_s(r, r', \theta - \theta') \, \omega(r, \theta) \, \omega(r', \theta') \, r' \, dr' \, d\theta' r \, dr \, d\theta,
\end{equation}
due to the symmetry and support properties of the function $\omega$, and where the kernel $\kappa_s$ is equal to
\begin{equation}
\label{def:kappa-s}
\kappa_s(r, r', \xi) := \sum_{n = 0}^{N - 1} \frac{c_s}{\Big(r^2 + r'^2 - 2 r r' \cos\big(\xi-\frac{2n\pi}{N}\big) \Big)^{1 - s}}.
\end{equation}
For further use, let us also define the function
$$
K_s \omega(r, \theta) := \int_S \kappa_s(r, r', \theta - \theta') \,
\omega(r', \theta') \, r' \, dr' \, d\theta'.
$$

\medskip

We consider the maximization problem under constraint
\begin{equation}
\label{def:max-prob-S}
\tag{\mbox{$\mathcal{P}_\lambda$}}
\cE_\lambda := \sup \Big\{ E_s(\omega) : \omega \in X_\lambda \text{ s.t. } M(\omega) = L(\omega) = N \Big\}.
\end{equation}

It is straightforward to check that the constraint set is non empty when
$\lambda$ is chosen large enough. Let $\varepsilon \equiv
\varepsilon(\lambda)$ be the length-scale defined by the identity
\begin{equation}
\label{def:varepsilon}
\lambda \pi \varepsilon^2 = 1.
\end{equation}
There exists (a unique) $1 - \varepsilon \leq r_\varepsilon
\leq 1 + \varepsilon$ such that the function $\varpi_\lambda$ given by
\begin{equation}
\label{eq:overline omega}
\varpi_\lambda := \sum_{n = 0}^{N - 1}\lambda \mathbbm{1}_{\cB \big( 
R_\frac{2n\pi}{N}(r_\varepsilon, 0), \varepsilon \big)},
\end{equation}
satisfies the constraints $M(\varpi_\lambda) = L(\varpi_\lambda) = N$. It is
clear also that this function belongs to $X_\lambda$ for
$\lambda$ large enough, that is for $\varepsilon$ small enough. 
\smallskip

Invoking the boundedness of the set $S$ and the Sobolev embedding theorem 
of $L^\infty(S)$ into $H^{- s}(S)$, it is also rather straightforward to conclude 
that the maximization problem~\eqref{def:max-prob-S} possesses a solution 
for all $\lambda$ large enough. We will prove 
\begin{Prop}
\label{prop:patch-exist}
Let $0 < s < 1$ be given, then for $\lambda$ sufficiently large there exists an angular Steiner symmetric function $\omega_\lambda \in X_\lambda$ such that
$$E_s(\omega_\lambda) = \cE_\lambda.$$
Moreover, there exist two numbers $\alpha_\lambda$ and $\mu_\lambda$ such that
\begin{equation}
\label{eq:omega-lambda}
\omega_\lambda = \lambda \mathbbm{1}_{\{ \psi_\lambda > 0 \}},
\end{equation}
where we have set
\begin{equation}
\label{def:psi-lambda-S}
\psi_\lambda := K_s \omega_\lambda + \frac{\alpha_\lambda}{2} |\cdot|^2 - \mu_\lambda.
\end{equation}
When $1/2 \leq s < 1$, the function $\psi_\lambda$ belongs to $W^{1, q}(S)$ for any real number $q > 1$, and it satisfies 
\begin{equation}
\label{eq:weak-formulation-S}
\int \omega_\lambda(x) \, \nabla^\perp \psi_\lambda(x) \cdot \nabla \varphi(x) \, dx = 0,
\end{equation}
for any function $\varphi \in \cC_c^\infty(S)$.
\end{Prop}

Note that at this level the weak formulation~\eqref{eq:weak-formulation-S} is only
known to be valid for test functions whose support is contained in $S$. This is
related to the aforementioned vanishing constraint of $\omega$ in $\mfS_N
\setminus S$, and will be dealt with in Section~\ref{sub:support} below.

\subsubsection{The penalized problems}

Our proof of Proposition~\ref{prop:patch-exist} does not rely on a direct
variational argument either. Because of the (eventually activated) $L^\infty$-constraint 
on $\omega$, it seems indeed difficult to derive~\eqref{eq:omega-lambda} and~\eqref{eq:weak-formulation-S}
in that way. Instead, we introduce the penalized functionals
\begin{equation}
\label{eq:penalized_energy}
E_{\lambda, p}(\omega) := E_s(\omega) - \frac{\lambda N}{p} \int_S \bigg( \frac{\omega}{\lambda} \bigg)^p,
\end{equation}
for any number $p > 1/s$, as well as the corresponding maximization
problems
\begin{equation}
\label{def:pen-max-prob}
\cE_{\lambda, p} := \sup \Big\{ E_{\lambda, p}(\omega) : \omega \in L_+^p(S) 
\text{ s.t. } M(\omega) = L(\omega) = 1 \Big\},
\end{equation}
where the functionals $M$ and $L$ are defined here as in~\eqref{def:C} and
\eqref{def:L} but with integration domain restricted to $S.$ 
Due to the boundedness of the set $S$, we observe that
$$\sup_{(r, \theta) \in S} \int_S \kappa_s(r, r', \theta - \theta')^q \, dr' \, d\theta' < \infty,$$
when $1 \leq q < 1/(1 - s)$. Hence it follows from the H\"older inequality that
both the terms in the right-hand side of~\eqref{eq:penalized_energy} are
well-defined when the vorticity $\omega$ lies in $L^p(S)$ with $p > 1/s$. In
this case, the quantities $M(\omega)$ and $L(\omega)$ are also well-defined, so
that the maximization problem~\eqref{def:pen-max-prob} makes sense. Moreover, we
can solve this problem as follows.

\begin{Lemma}
\label{lem:penalized-max}
Let $0 < s < 1$, $\lambda > 0$ and $p > 1/s$. Denote by $p'$ the H\"older
conjugate of $p$ defined by $1/p + 1/p' = 1$. There exists a function
$\omega_{\lambda, p} \in L_+^p(S)$ such that
\begin{equation}
\label{eq:pen-max-prob}
E_{\lambda, p}(\omega_{\lambda, p}) = \max \Big\{ E_{\lambda, p}(\omega) : \omega \in L_+^p(S) 
\text{ s.t. } M(\omega) = L(\omega) = 1 \Big\}.
\end{equation}
The function $\omega_{\lambda, p}$ is angular Steiner symmetric and bounded.
Moreover, there exist two numbers $\alpha_{\lambda, p}$ and $\mu_{\lambda, p}$
such that the function $\omega_{\lambda, p}$ can be written as
\begin{equation}
\label{eq:max-prob-eq}
\omega_{\lambda ,p} = \lambda(\psi_{\lambda, p})_+^{p' - 1},
\end{equation}
with
\begin{equation}
\label{eq:max-prob-psi}
\psi_{\lambda, p} := K_s \, \omega_{\lambda, p} + \frac{\alpha_{\lambda, p}}{2} |\cdot|^2 - \mu_{\lambda, p}.
\end{equation}
\end{Lemma}

The proof of Lemma~\ref{lem:penalized-max} is standard. Existence follows from
bounding the $L^p$-norm of the minimizing sequences corresponding
to~\eqref{def:pen-max-prob} and from applying standard weak compactness results.
The angular Steiner symmetry of the function $\omega_{\lambda, p}$ is a
consequence of Lemmas~\ref{lem:inv-C-L} and~\ref{lem:symmetry}.
Equations~\eqref{eq:max-prob-eq} and~\eqref{eq:max-prob-psi} are no more than
the Euler-Lagrange equations of the maximization
problem~\eqref{def:pen-max-prob}. In particular, the numbers $\alpha_{\lambda,
p}$ and $\mu_{\lambda, p}$ are interpreted as the Lagrange multipliers of the
problem. We refer to Subsection~\ref{sub:penalized-max} for more details and the
proofs.

\subsubsection{Bounds on the maximizers $\omega_{\lambda, p}$ and the Lagrange multipliers $\alpha_{\lambda, p}$ and $\mu_{\lambda, p}$}

Our goal is now to construct solutions of the maximization problems~\eqref{def:max-prob-S} as limits when $p \to \infty$ of the functions $\omega_{\lambda, p}$ of Lemma~\ref{lem:penalized-max}. Before passing to this limit, we establish that the Lagrange multipliers $\alpha_{\lambda, p}$ and $\mu_{\lambda, p}$ are bounded uniformly with respect to $p$.

\begin{Lemma}
\label{lem:alpha_mu_bounded}
Let $0 < s < 1$ and $\lambda > 0$. There exists a positive number $p_0 > 1/s$ such that there exists a positive number $C_\lambda$ independent of $p$ for which the Lagrange multipliers $\alpha_{\lambda, p}$ and $\mu_{\lambda, p}$ in Lemma~\ref{lem:penalized-max} satisfy
$$|\alpha_{\lambda, p}| + |\mu_{\lambda, p}| \leq C_\lambda,$$
for any $p \geq p_0$.
\end{Lemma}

This lemma follows from a uniform bound on the function $K_s \omega_{\lambda, p}$. As a consequence of the integrability properties of the kernel $K_s$, there exists a positive number $C_\lambda$ independent of $p$ such that
$$\big\| K_s \omega_{\lambda, p} \big\|_{L^\infty} \leq C_\lambda.$$
With this inequality at hand, we can control the support of the function $\psi_{\lambda, p}$ and then of the function $\omega_{\lambda, p}$. In case the Lagrange multipliers $\alpha_{\lambda, p}$ and $\mu_{\lambda, p}$ are not bounded independently of $p$, this control is enough to establish a contradiction with the constraints $M(\omega_{\lambda, p}) = L(\omega_{\lambda, p}) = 1$.

We next apply standard regularity estimates to the Euler-Lagrange equations~\eqref{eq:max-prob-eq} and~\eqref{eq:max-prob-psi} in order to bound the functions $\omega_{\lambda, p}$ and $\psi_{\lambda, p}$ uniformly with respect to $p$. When $1/2 \leq s < 1$, the functions $\omega_{\lambda, p}$ and $\psi_{\lambda, p}$ own sufficient smoothness so as to satisfy the weak formulation of~\eqref{eq:rotating frame dynamics} in~\eqref{eq:weak-formulation-S}. As a matter of fact, this weak formulation follows from the collinearity of the gradients $\nabla \omega_{\lambda, p}$ and $\nabla \psi_{\lambda, p}$, which in turn is a consequence of~\eqref{eq:max-prob-eq} provided that $\omega_{\lambda, p}$ and $\psi_{\lambda, p}$ are smooth enough. More precisely, we show

\begin{Lemma}
\label{lem:omega-bounded}
Let $0 < s < 1$, $\lambda > 0$ and $p \geq p_0$, where $p_0$ is defined in Lemma~\ref{lem:alpha_mu_bounded}. Consider the solution $\omega_{\lambda, p}$ to the maximization problem~\eqref{eq:pen-max-prob} constructed in Lemma~\ref{lem:penalized-max}, and set
$$\psi_{\lambda, p} = K_s \omega_{\lambda, p} + \frac{\alpha_{\lambda, p}}{2} |\cdot|^2 - \mu_{\lambda, p},$$
where $\alpha_{\lambda, p}$ and $\mu_{\lambda, p}$ are the corresponding Lagrange multipliers. There exists a positive number $C_\lambda$ independent of $p$ such that
\begin{equation}
\label{eq:bound-omega}
\| \omega_{\lambda, p} \|_{L^r} \leq C_\lambda,
\end{equation}
for any $1 \leq r \leq \infty$, while there exist positive numbers $C_{\lambda, r}(R)$, not depending on $p$, such that
\begin{equation}
\label{eq:bound-psi}
\| \psi_{\lambda, p} \|_{W^{2 s, r}(B(0, R))} \leq C_{\lambda, r}(R),
\end{equation}
for any $1 < r < \infty$ and any positive number $R$. In particular, when $1/2 \leq s < 1$, the functions $\omega_{\lambda, p}$ and $\psi_{\lambda, p}$ satisfy the weak equation
$$\int_S \omega_{\lambda, p}(x) \, \nabla^\perp \psi_{\lambda, p}(x) \cdot \nabla \varphi(x) \, dx = 0,$$
for any function $\varphi \in \cC_c^\infty(S)$.
\end{Lemma}

\subsubsection{Convergence in the limit $p \to \infty$ towards vortex patch solutions}

With Lemmas~\ref{lem:alpha_mu_bounded} and~\ref{lem:omega-bounded} at hand, we
are in position to complete the proof of Proposition~\ref{prop:patch-exist}.
Since the Lagrange multipliers $\alpha_{\lambda, p}$ and $\mu_{\lambda, p}$, as
well as the functions $\omega_{\lambda, p}$ and $\psi_{\lambda, p}$, are
uniformly bounded with respect to $p$, we can invoke a compactness argument in
order to take the limit $p \to \infty$. This eventually provides a limit
function $\omega_\lambda \in X_\lambda$, which maximizes the energy $E_s$ under
the constraints $M(\omega_\lambda) = L(\omega_\lambda) = 1$. We next derive from
the following lemma that this function is a vortex patch. 

\begin{Lemma}[Bathtub principle for Riesz integrals]
\label{lem:Riesz-bathtub}
Let $\mu$ and $\nu$ be two positive numbers. Given a positive number $\eta$, set
$$\cG_\eta(\RR) := \Big\{ g \in L^\infty(\RR, [0, 1]) : \int_\RR g \leq \eta \Big\},$$
and consider a function $f \in L^1(\RR)$, which is even and non-increasing on $\RR_+$. 
Then the maximization problem
\begin{equation}
\label{eq:double_integral_to_maximize}
\cI_{\mu, \nu} := \sup_{(g, h) \in \cG_\mu(\RR) \times \cG_\nu(\RR)} 
\int_\RR \int_\RR f(x - y) \, g(x) \, h(y) \, dx \, dy
\end{equation}
is solved by the functions
$$
	g = \mathbbm{1}_{\big[ - \frac{\mu}{2}, \frac{\mu}{2} \big]} \quad 
	\text{and} \quad h = \mathbbm{1}_{\big[ - \frac{\nu}{2},\frac{\nu}{2} \big]}.
$$
Moreover, when the function $f$ is decreasing on $\RR_+$, this solution is
unique (up to a translation).
\end{Lemma}

The proof of this lemma consists in applying the bathtub principle (see
e.g.~\cite[Theorem 1.14]{LiebLos0}) to the double integral in the right-hand
side of~\eqref{eq:double_integral_to_maximize}. Double integrals taking this
form are usually called Riesz integrals in reference to the Riesz rearrangement
inequality~\cite{Riesz1}. For fixed numbers $r$ and $r'$, the integrals
depending on the variables $\theta$ and $\theta'$ in the
expression~\eqref{eq:polar-Es} of the energy $E_s$ are exactly of this form. In
particular, it suffices to apply Lemma~\ref{lem:Riesz-bathtub} to them in order
to prove that the functions $\omega_\lambda$ are the characteristic functions of
measurable sets $\Omega_\lambda$.

The proof of Proposition~\ref{prop:patch-exist} then reduces to establish that
these sets are equal to the upper level sets $\{ \psi_\lambda > 0 \}$. In the
limit $p \to \infty$, we deduce from~\eqref{eq:max-prob-eq} that
$$\{ \psi_\lambda > 0 \} \subset \Omega_\lambda \subset \{ \psi_\lambda \geq 0 \}.$$
It then results from the monotonicity properties of the kernel $\kappa_s$ given
by Lemma~\ref{lem:combinatorics} that the functions $\theta \mapsto
\psi_\lambda(r, \theta)$ are decreasing on $(0, \pi/2 N)$ for any fixed number
$\frac{1}{2} < r < 2$. As a consequence, the level sets $\{ \psi_\lambda = 0 \}$ have
measure zero, which eventually gives~\eqref{eq:omega-lambda}.

Note also that the weak equation~\eqref{eq:weak-formulation-S} follows as in the
proof of Lemma~\ref{lem:omega-bounded} from the fact that $\psi_\lambda$ belongs
to the Sobolev spaces $W^{1, r}(S)$ for any $1 < r < \infty$ when $1/2 \leq s <
1$. We refer to Subsection~\ref{sub:prop-patch-exist} for more details and
the proofs.

\subsection{Description of the vortex patch support}
\label{sub:support}

The purpose of this section is to complete the description
of the support of the vorticity $\omega_\lambda$ obtained in
Proposition~\ref{prop:patch-exist}. We will eventually 
show in Proposition~\ref{prop:condsupp} that it is entirely 
contained in $N$ small balls centered at the vertices of a 
regular polygon, at least when $\lambda$ is sufficiently large. 
Combining Proposition~\ref{prop:patch-exist} with Proposition
\ref{prop:condsupp}, the proof of Theorem 
\ref{thm:patch-exist} will then easily follow in Section 
\ref{sec:completed}.

\medskip

In a first step, we consider functions $\omega$ that
are merely of the form
\begin{equation}
\label{eq:arbpatch}
\omega = \lambda \mathbbm{1}_\Omega, \ \Omega \subset S \text{ measurable,} \ M(\omega) = 1.
\end{equation}
Clearly, $\omega_\lambda \mathbbm{1}_{S}$ is such a function. For arbitrary measurable sets $X$,
$X'$ in $\RR^2$, we next define the localized energy
$$
I_s(\omega, X, X') := \int_{X} \int_{X'} \bigg( \frac{\varepsilon}{|x - x'|} \bigg)^{2 (1 - s)} 
\omega(x')\omega(x)dx'dx 
$$
and the localized concentration of mass
$$
M(\omega, X) :=\int_{X} \omega(x)dx.
$$
We also let 
\begin{equation}
\label{eq:define I_s}
\cI_s := \frac{1}{\pi^2} \int_{\cB(0, 1)} \int_{\cB(0, 1)} \frac{dx \, dx'}{|x -
x'|^{2 (1 - s)}} = I_s(\lambda \mathbbm{1}_{\cB(0,\varepsilon)},\RR^2, \RR^2).
\end{equation}
The quantity $\cI_s$ is positive and elementary computations show that $\frac{1}{6} \leq 
s\cI_s \leq 1$ for all $s \in (0,1),$ in particular it is bounded and bounded
away from zero independently of $s.$ 

We infer the following from the Riesz rearrangement inequality~\cite{Riesz1}. This
provides lower bounds on the localized concentration of mass in terms of the localized
energy.

\begin{Lemma}
\label{lem:energy}
Assume that $\omega$ satisfies~\eqref{eq:arbpatch} and $X, X' \subset \RR^2$ are
measurable. Then
\begin{align*}
&I_s(\omega, X, X)	\leq \cI_s \, M(\omega, X)^{1 + s},\\
&I_s(\omega, X, X') \leq \frac{1}{s} M(\omega,X)M(\omega,X')^s,\\
&I_s(\omega, X, X') \leq \bigg( \frac{\varepsilon}{\mathrm{dist}(X, X')}\bigg)^{2 (1 -
s)} M(\omega, X) M(\omega, X'),
\end{align*}
where the last inequality is only meaningful provided that 
$$\mathrm{dist}(X,X'):=\inf_{x\in X}\inf_{x'\in X'}|x-x'|$$ 
has positive value.
\end{Lemma}

Being a solution of the minimisation problem~\eqref{def:max-prob-S}, 
we can derive some first tight local energy estimates for $\omega_\lambda$
using a comparison argument with the test function $\varpi_\lambda$ introduced
in~\eqref{eq:overline omega}.

\begin{Lemma}
\label{lem:E-s-bounds}
Let $\omega_\lambda$ be a solution to~\eqref{def:max-prob-S}. There exists a positive number $C$, 
depending only on $s$, such that the bounds
\begin{equation}
\label{eq:rescaled energy}
\cI_s \leq I_s(\omega, S, S) \leq \cI_s + C
\varepsilon^{2 (1 - s)}
\end{equation}
hold for $\omega := \omega_\lambda \mathbbm{1}_S.$
\end{Lemma}

For functions of the form~\eqref{eq:arbpatch} and for which~\eqref{eq:rescaled
energy} holds, one can obtain a first weak form of concentration by an
elementary covering argument.

\begin{Lemma}
\label{lem:Iam not small}
There exist positive numbers $\Lambda_0$ and $\eta$ depending only on $s$ such that if 
$\omega$ is of the form~\eqref{eq:arbpatch} and satisfies the bounds 
\eqref{eq:rescaled energy}, then
there exists $x \in S$ such that
\begin{equation}
\label{eq:define x^*}
M(\omega, \cB(x,\Lambda_0 \varepsilon)) \geq \eta > 0.
\end{equation}
\end{Lemma}

Combining Lemmas~\ref{lem:energy},~\ref{lem:E-s-bounds}
and~\ref{lem:Iam not small} with a more elaborate 
convexity and combinatory argument, we then show that splitting 
vorticity is not favourable and obtain a decay estimate for the 
vorticity density.

\begin{Lemma}
\label{lem:decay}
There exists a constant $C>0$, depending only on $s$, such that 
if $\omega$ is of the form~\eqref{eq:arbpatch} and satisfies the bounds 
\eqref{eq:rescaled energy}, then
$$
M\big(\omega, S\setminus\cB\big((1,0), \Lambda\varepsilon\big)\big)
\leq \frac{C}{\Lambda^{\gamma_s}}, \qquad \forall \Lambda > 0,
$$
where the decay rate $\gamma_{s}$ is given by $\gamma_s := (1 + \frac{1}{2 (1 -
s)})^{-1}.$
\end{Lemma}

At this point we have only used arguments which rely on tight energy
bounds for $\omega.$ It is hopeless to expect that complete concentration
of vorticity such as the one stated in $(iii)$ of Theorem 
\ref{thm:patch-exist} could hold under such assumptions only (it is indeed
not complicated to build counter-examples). 
For the next step of the argument, we therefore restrict our focus to the
solutions $\omega_\lambda$ of the maximization problem
\eqref{def:max-prob-S}, in order to exploit the equations
\eqref{eq:omega-lambda}-\eqref{def:psi-lambda-S}. 

Note that~\eqref{eq:omega-lambda}-\eqref{def:psi-lambda-S} only implies a very weak 
form of regularity, especially when $s$ is small. For that reason, we rely on 
finite differences of $\psi_\lambda$ at different pairs of well-chosen points 
inside and outside of $\supp(\omega_\lambda)$ and combine the resulting identities 
with integral estimates based on the decay obtained in Lemma~\ref{lem:decay}. 
The first localization result that we prove is the following.

\begin{Lemma}
\label{lem:jen ai marre}
There exists a positive constant $\Lambda$ depending only on $s$ such that if 
$\omega_\lambda$ is a solution of the maximization 
problem~\eqref{def:max-prob-S} then
\begin{equation}
\label{eq:on arrete quand}
{\rm supp}(\omega_\lambda) \cap \mfS_N \;\subseteq\; \cB((1,0), \Lambda\varepsilon)\cup B,
\end{equation}
where
$$B:=\big\{(r,\theta)\in S:\;2-1/8\leq r\leq2\;\;\mathrm{or}\;\;1/2\leq r\leq1/2+1/8\big\}\;\subseteq S.$$
\end{Lemma}

As explained with more details in the proof, it is not possible to eliminate a
possible residual accumulation of mass far from $\cB((1,0), \Lambda\varepsilon)$
with arguments using finite differences of $\psi_\lambda$. At this stage of the
proof, the case $M(\omega_\lambda,B)>0$ (even very small) cannot be excluded.
Nevertheless, there is a gap of order $1$ (with respect to $\varepsilon$) between
the two sets $\cB(x, \Lambda\varepsilon)$ and $B$ and then, using again energetical
arguments, we can prove that $M(\omega_\lambda,B)>0$ contradicts the maximality 
of $\omega_\lambda$.
Here is the conclusive localization result.

\begin{Prop}
\label{prop:condsupp}
There exists a positive constant $\Lambda$, depending only on $s$, such that 
if $\omega_\lambda$ is a solution of the maximization 
problem~\eqref{def:max-prob-S} then 
$$
{\rm supp}(\omega_\lambda) \cap \mfS_N \subseteq \cB((1,0), \Lambda\varepsilon).
$$
In particular, if $\lambda$ is sufficiently large we have
\begin{equation}
\label{eq:distance au bord}
\dist \big( \supp(\omega_\lambda), \ \mfS_N \setminus S \big) > 0,
\end{equation}
and therefore for $1/2 \leq s < 1$ the function $\omega_\lambda$ is a weak solution 
to~\eqref{eq:rotating frame dynamics} in the sense of~\eqref{eq:weak
formulation}.
\end{Prop}

\subsection{Proof of Theorem~\ref{thm:patch-exist} completed}
\label{sec:completed}

Combining Propositions~\ref{prop:patch-exist} and~\ref{prop:condsupp},
we deduce that statements $(i)$, $(ii)$ and
$(iii)$ of Theorem~\ref{thm:patch-exist} hold. It only remains 
to establish statements $(iv)$ and $(v)$, which concern the asymptotic as
$\lambda$ tends to $+\infty,$ that is~\eqref{eq:lim_omega}, 
\eqref{eq:lim_alpha} and~\eqref{eq:mu behavior}. 

The convergence in~\eqref{eq:lim_omega} is immediate. It follows
from the constraints $M(\omega_\lambda) = L(\omega_\lambda) = N$, 
the rotational and Steiner symmetry assumptions, and Proposition
\ref{prop:condsupp}.

Concerning the convergence in~\eqref{eq:lim_alpha}, it can be
computed exactly using the equations~\eqref{eq:patch} and
\eqref{def:psi-lambda} for $\omega_\lambda$. This is the reason 
why we need to assume~\footnote{We do not know if this is 
only technical.} that $\frac{1}{2} \leq s < 1.$

Let $y = (y_1, y_2) \in \RR^2$ be arbitrary, and consider a test 
function $\varphi \in \cC^\infty(S)$ such that
$$
\varphi(x) = \frac{|x|^2}{2} + y \cdot x, \qquad \forall x \in \supp(\omega_\lambda).
$$
From~\eqref{eq:weak formulation} we deduce therefore that 
$$
\alpha_\lambda T_{1,\lambda,y} = c_s(1-s)T_{2,\lambda,y},
$$
where 
\begin{align*}
&T_{1,\lambda,y} = \int_S \omega_\lambda(x) \, x^\perp \cdot y \, dx,\\
&T_{2,\lambda,y} = \int_S \int_S \omega_\lambda(x) \, \omega_\lambda(x') \, 
\sum_{n = 1}^{N - 1} \frac{\big(x - R_\frac{2 n \pi}{N} x' \big)^\perp}{\big| x 
- R_\frac{2 n \pi}{N} x' \big|^{4 - 2s}} \cdot \Big( y - R_\frac{2 n \pi}{N}
y \Big) \, dx \, dx'.
\end{align*}
In view of~\eqref{eq:patch} and statement $(iii)$ of Theorem
\ref{thm:patch-exist}, we can pass to the limit 
$\lambda \to +\infty$ in the above two terms and obtain 
$$
T_{1,\lambda,y} \longrightarrow y_2,\qquad \text{ and } \qquad T_{2,\lambda,y} 
\longrightarrow \sum_{n = 1}^{N - 1} \frac{\big( (1, 0) 
- R_\frac{2 n \pi}{N}(1, 0) \big)^\perp}{\big| (1, 0) - R_\frac{2 n \pi}{N}(1,0) 
\big|^{4 - 2 s}} \cdot \Big( y - R_\frac{2 n \pi}{N} y \Big).
$$
Taking $y = (0, 1) = (1, 0)^\perp$, the conclusion~\eqref{eq:lim_alpha} follows. 

Finally we consider~\eqref{eq:mu behavior}, whose proof bares resemblance with
the one of Proposition~\ref{prop:condsupp}. 
We first take an arbitrary point $y \in \cB((1,0), 2\varepsilon) 
\setminus \Omega_\lambda$ (such a point exists since by definition $\varepsilon$ 
is the radius of a ball with the same area $1/\lambda$ as $\Omega_\lambda$), so that
$\psi_\lambda(y) \leq 0$, and therefore 
\begin{equation}
\label{eq:lower estimate on mu}
K_s\omega_\lambda(y) + \frac{\alpha_\lambda}{2} |y|^2 \leq \mu_\lambda.
\end{equation}
On the other hand, from $(iii)$ we infer that
\begin{align*}
K_s &\omega_\lambda(x) \geq \int_{S} 
\frac{\omega_\lambda(x')}{|x - x'|^{2 (1 - s)}} \, dx' - C 
\\&=\int_{\cB((1, 0), C\sqrt{\pi}\varepsilon)} 
\frac{\omega_\lambda(x')}{|x - x'|^{2 (1 - s)}} \, dx' - C 
\geq \frac{1}{C \lambda^{1 - s}} - C.
\end{align*}
Combining both inequalities and using the fact that $\alpha_\lambda$ is already
known to be uniformly bounded by~\eqref{eq:lim_alpha}, we deduce the lower
estimate in~\eqref{eq:mu behavior}. Concerning the upper estimate, we consider
instead a point $x \in \cB((1, 0), C/\sqrt{\lambda}) \cap \Omega_\lambda.$ Then
$\psi_\lambda(y) \geq 0$, and therefore 
\begin{equation}
\label{eq:upper estimate on mu}
\mu_\lambda \leq K_s \omega_\lambda(x) + \frac{\alpha_\lambda}{2} |x|^2.
\end{equation}
On the other hand by rearrangement
\begin{align*}
K_s \omega_\lambda(x) \leq 
\int_{\cB((1, 0), 2C/\sqrt{\lambda})} 
\frac{\omega_\lambda(x')}{|x - x'|^{2 (1 - s)}} \, dx' + C 
\leq \int_{\{|x'| \leq \varepsilon\}} 
\frac{\lambda}{|x'|^{2(1 - s)}} \, dx' + C 
\leq \frac{C}{\lambda^{1 - s}} + C,
\end{align*}
and the upper estimate for $\mu_\lambda$ follows similarly. This ends the proof of
Theorem~\ref{thm:patch-exist}. \qed

\section{Details of the proofs}
\label{sec:proofs}

We present in this section, in the order of appearance, the proofs of 
all the claims that were postponed from Section~\ref{sec:strategy}.

\subsection{Proof of Lemma~\ref{lem:penalized-max}}
\label{sub:penalized-max}

We first establish the existence of the maximizing vorticity $\omega_{\lambda,
p}$. Recall that the kernel $\kappa_s$ is defined so as to satisfy
$$
\kappa_s(r, r', \theta - \theta') = \sum_{n = 0}^{N - 1} \frac{c_s}{\Big(r^2 + r'^2 - 
2 r r' \cos\big(\theta-\theta'-\frac{2n\pi}{N}\big) \Big)^{1 - s}},
$$
for $x = (r \cos(\theta), r \sin(\theta))$ and $x' = (r' \cos(\theta'), r' \sin(\theta'))$. Hence, we derive from the boundedness of $S$ that
\begin{equation}
\label{eq:estim-ks} 
\sup_{(r, \theta) \in S} \int_S \kappa_s(r, r', \theta - \theta')^{p'} r' \, dr' \, d\theta' < \infty,
\end{equation}
under the condition $p > 1/s$. Given any function $\omega \in L^p(S)$, we infer from the H\"older inequality that
\begin{equation}
\label{eq:estim-Ks}
\| K_s\omega \|_{L^\infty} \leq C \| \omega \|_{L^p},
\end{equation}
where the number $C$ only depends on $s$. The energy $E_s$ is then controlled by
\begin{equation}
\label{eq:estim_energy}
E_s(\omega) \leq \frac{C N}{2} \| \omega \|_{L^1} \| \omega \|_{L^p} = \frac{C N}{2} \|\omega\|_{L^p},
\end{equation}
under the conditions $\omega \geq 0$ and $M(\omega) = 1$. We next apply the Young inequality in order to obtain
$$E_s(\omega) \leq \Big( \frac{C \lambda^\frac{1}{p'} N^\frac{1}{p'}}{2} \Big) \bigg( \frac{N^\frac{1}{p} \| \omega \|_{L^p}}{\lambda^\frac{1}{p'}}\bigg) \leq K_{\lambda, p} + \frac{\lambda N}{p} \int_S \Big( \frac{\omega}{\lambda} \Big)^p.$$
Here, the constant $K_{\lambda, p} := C^{p'} \lambda N/(2^{p'} {p'})$ is uniformly bounded with respect to $p \to \infty$, and it controls the penalized energy as
\begin{equation}
\label{eq:bounded_penalized_energy}
E_{\lambda, p}(\omega) \leq K_{\lambda, p}.
\end{equation}
In particular, the value of the supremum $\cE_{\lambda, p}$ in~\eqref{def:pen-max-prob} is finite.

In another direction, we also deduce from~\eqref{eq:estim_energy} that
$$\frac{\lambda^{1 - p} N}{p} \| \omega \|_{L^p}^p \leq \frac{C N}{2} \| \omega \|_{L^p} - E_{\lambda, p}(\omega) \leq\frac{C^{p'} \lambda N}{2 p'} + \frac{\lambda^{1 - p} N}{2 p} \| \omega \|_{L^p}^p - E_{\lambda, p}(\varpi_\lambda),$$
when $E_{\lambda, p}(\varpi_\lambda) \leq E_{\lambda, p}(\omega)$. Here, the notation $\varpi_\lambda$ refers as before to the function in~\eqref{eq:overline omega}. We next obtain
\begin{equation}
\label{eq:bound_Lp}
\| \omega \|_{L^p}^p \leq (p - 1) \lambda^p C^{p'} - \frac{2 p \lambda^{p - 1}}{N} E_{\lambda, p}(\varpi_\lambda),
\end{equation}
so that there exists a maximizing sequence $(\omega_j)_{j \in \NN}$ for~\eqref{def:pen-max-prob}, which is bounded in $L^p(S)$. Hence, there exists a function $\omega_{\lambda, p} \in L^p(S)$ such that, up to a subsequence, $(\omega_j)_{j \in \NN}$ weakly tends to $\omega_{\lambda, p}$ in $L^p(S)$. This convergence is enough to guarantee that $\omega_{\lambda, p} \geq 0$ a.e., as well as the identities $M(\omega_{\lambda, p}) = 1$ and $L(\omega_{\lambda, p}) = 1$. Moreover,
\begin{equation}
\label{eq:lim-pen-E}
\cE_{\lambda, p} = \limsup_{j \to \infty} E_{\lambda, p}(\omega_j) \leq E_{\lambda, p}(\omega_{\lambda, p}).
\end{equation}
Indeed, estimate~\eqref{eq:estim-ks} implies that $\kappa_s\in L^{p'}(S^2)$, and the functions $(x, y) \mapsto \omega_j(x) \omega_j(y)$ weakly tend to the function $(x, y) \mapsto \omega_{\lambda, p}(x) \omega_{\lambda, p}(y)$ in $L^p(S^2)$ as $j \to \infty$. Therefore, the quantities $E_s(\omega_j)$ converge to $E_s(\omega_{\lambda, p})$, and inequality~\eqref{eq:lim-pen-E} follows from the weak lower semi-continuity of the $L^p$-norm. Finally, this inequality guarantees that the limit vorticity $\omega_{\lambda, p}$ solves the maximization problem~\eqref{def:pen-max-prob}. Moreover, we can invoke Lemmas~\ref{lem:inv-C-L} and~\ref{lem:symmetry} in order to claim that the vorticity $\omega_{\lambda, p}$ is angular Steiner symmetric.

We now turn to the proof of~\eqref{eq:max-prob-eq} and~\eqref{eq:max-prob-psi}. These two equations correspond to the Euler-Lagrange equations of the maximization problem~\eqref{def:pen-max-prob}. In order to write them, we observe that, given any measurable subset $X$ of $S$ with positive measure, the non-vanishing linear forms $M$ and $L$ are linearly independent on $L^\infty(X)$. Otherwise, there exists a number $a \neq 0$ such that
$$\int_X (a - |x|^2) \varphi(x) \, dx = 0,$$
for any $\varphi \in L^\infty(X)$. In particular, the function $x \mapsto a - |x|^2$ vanishes almost everywhere on $X$, which is only possible when $a$ is positive and $X$ is a subset of the circle of center $(0, 0)$ and radius $ \sqrt{a}$. Since this circle is of measure $0$, this contradicts the fact that $X$ has positive measure. As a consequence of the linear independence of $M$ and $L$ on $L^\infty(X)$, we conclude that the linear mapping $\varphi \mapsto (M(\varphi), L(\varphi))$ is onto $\RR^2$.

Consider next a positive number $\delta$ such that $X_\delta = \{ \omega_{\lambda, p} \geq \delta \}$ has positive measure. Note that all the numbers $\delta$ small enough satisfy this assumption, otherwise the vorticity $\omega_{\lambda, p}$ vanishes, which contradicts the condition $M(\omega_{\lambda, p}) = 1$. Taking into account the previous arguments, we derive the existence of two functions $\varphi_1$ and $\varphi_2$ in $L^\infty(X_\delta)$ with $M(\varphi_1) = L(\varphi_2) = 1$ and $M(\varphi_2) = L(\varphi_1) = 0$. In particular, given any function $\zeta \in L^\infty(S)$ such that $\zeta \geq 0$ on $S \setminus X_\delta$, we can introduce the function
$$\omega_h = \omega_{\lambda, p} + h \big( \zeta - M(\zeta) \varphi_1 - L(\zeta) \varphi_2 \big),$$
for any positive number $h$. When this number is small enough, we check that this function belongs to $L_+^p(S)$, with $M(\omega_h) = L(\omega_h) = 1$. The maximizing nature of the vorticity $\omega_{\lambda, p}$ then gives
$$0 \geq \lim_{h \to 0^+} \frac{d}{dh} \Big( E_{\lambda, p}(\omega_h) \Big) = E'_{\lambda, p}(\omega_{\lambda, p}) \, \zeta - M(\zeta) E'_{\lambda, p}(\omega_{\lambda, p}) \, \varphi_1 - L(\zeta) E'_{\lambda, p}(\omega_{\lambda, p}) \, \varphi_2.$$
In this expression, the differential $E'_{\lambda, p}(\omega)$ is given by
$$E'_{\lambda, p}(\omega) \zeta = N \int_S \zeta \bigg( K_s \omega - \Big( \frac{\omega}{\lambda} \Big)^\frac{1}{{p'} - 1} \bigg).$$
Setting $\mu_{\lambda, p} = E'_{\lambda, p}(\omega_{\lambda, p}) \, \varphi_1/N$, $\alpha_{\lambda, p} = - 2 E'_{\lambda, p}(\omega_{\lambda, p}) \, \varphi_2/N$ and $\psi_{\lambda, p}(x) := K_s \omega_{\lambda, p}(x) + \alpha_{\lambda, p} |x|^2/2 \linebreak[0] - \mu_{\lambda, p}$ as in~\eqref{eq:max-prob-psi}, we obtain
$$\int_S \zeta \bigg( \psi_{\lambda, p} - \Big( \frac{\omega_{\lambda, p}}{\lambda} \Big)^\frac{1}{{p'} - 1} \bigg) \leq 0.$$
Going back to the condition $\zeta \geq 0$ on $S \setminus X_\delta$, we infer that
$$\psi_{\lambda, p} = \Big( \frac{\omega_{\lambda, p}}{\lambda} \Big)^\frac{1}{p' - 1} \text{ on } X_\delta, \quad \text{and} \quad \psi_{\lambda, p} \leq \Big( \frac{\delta}{\lambda} \Big)^\frac{1}{p' - 1} \text{ on } S \setminus X_\delta.$$
Equation~\eqref{eq:max-prob-eq} follows by taking the limit $\delta \to 0$.

Finally, it follows from~\eqref{eq:estim-Ks} that the function $\psi_{\lambda, p}$ is bounded on $S$. In view of~\eqref{eq:max-prob-eq}, so is the vorticity $\omega_{\lambda, p}$. This completes the proof of Lemma~\ref{lem:penalized-max}. \qed

\subsection{Proof of Lemma~\ref{lem:alpha_mu_bounded}}

The proof relies on a uniform bound on the function $K_s \omega_{\lambda, p}$, which is derived from estimates~\eqref{eq:estim-Ks} and~\eqref{eq:bound_Lp}. Indeed, we deduce from the definition of the function $\varpi_\lambda$ that
$$E_{\lambda, p}(\varpi_\lambda) = E_s(\varpi_\lambda) - \frac{N}{p}.$$
Since $p > 1/s$, this gives
$$\big| E_{\lambda, p}(\varpi_\lambda) \big| \leq C_\lambda,$$
where $C_\lambda$ denotes, here as in the sequel, a positive number not depending on $p$ and possibly changing from line to line. Hence, we infer from~\eqref{eq:estim-Ks} and~\eqref{eq:bound_Lp} that
\begin{equation}
\label{eq:Ks_omega_borne}
\| K_s \omega_{\lambda, p} \|_{L^\infty} \leq C_\lambda.
\end{equation}
As a consequence of its definition, the function $\psi_{\lambda, p}$ then satisfies
\begin{equation}
\label{eq:encadre_psi}
\frac{\alpha_{\lambda, p}}{2} r^2 - \mu_{\lambda, p} - C_\lambda \leq \psi_{\lambda, p}(r, \theta) \leq \frac{\alpha_{\lambda, p}}{2} r^2 - \mu_{\lambda, p} +C_\lambda.
\end{equation}
If the Lagrange multipliers $\alpha_{\lambda, p}$ and $\mu_{\lambda, p}$ were not bounded independently of $p$, these estimates would provide a contradiction with the constraints $M(\omega_{\lambda, p}) = L(\omega_{\lambda, p}) = 1$.

More precisely, let us now assume that
\begin{equation}
\label{eq:mu-infini}
\mu_{\lambda, p} \to + \infty,
\end{equation}
as $p \to \infty$. In this case, we can suppose that $\mu_{\lambda, p} > C_\lambda$ for $p$ large enough. It then follows from inequalities~\eqref{eq:encadre_psi} that $\alpha_{\lambda, p} > 0$. Otherwise, the function $\psi_{\lambda, p}$ is non-positive, which eventually contradicts the constraint $M(\omega_{\lambda, p}) = 1$.
Consider next the number
$$r_0(\psi_{\lambda, p}) := \inf_{- \pi/4 \leq \theta \leq \pi/4} \Big( \infess \big( \supp(r \mapsto \big( \psi_{\lambda, p})_+(r,\theta) \big) \big) \Big).$$
Since the function $(\psi_{\lambda, p})_+$ is angular Steiner symmetric, this infimum is achieved for $\theta = 0$. Observe that the inequality $f \geq g$ implies that $r_0(f) \leq r_0(g)$. Therefore, we derive from~\eqref{eq:encadre_psi} that
$$r_0(\omega_{\lambda, p}) = r_0(\psi_{\lambda, p}) \geq r_0 \Big( \frac{\alpha_{\lambda, p}}{2} |\cdot|^2 - \mu_{\lambda, p} + C_\lambda \Big) = \sqrt{\frac{2}{\alpha_{\lambda, p}} (\mu_{\lambda, p} - C_\lambda)} := r_0^\ast.$$
Similarly, we obtain
$$r_0(\omega_{\lambda, p}) = r_0(\psi_{\lambda, p}) \leq r_0 \Big( \frac{\alpha_{\lambda, p}}{2} |\cdot|^2 - \mu_{\lambda, p} - C_\lambda \Big) = \sqrt{\frac{2}{\alpha_{\lambda, p}} (\mu_{\lambda, p} + C_\lambda)} := r_1^\ast,$$
and we notice that
\begin{equation}
\label{eq:r_0_r_1_equivalent}
\frac{r_0^\ast}{r_1^\ast} = \sqrt{\frac{\mu_{\lambda, p} - C_\lambda}{\mu_{\lambda, p} + C_\lambda}} \to 1.
\end{equation}
when $p \to \infty$. Given a positive number $\delta < 1$, we now face two possibilities.

\begin{Case}
$\limsup_{p \to \infty} r_1^\ast \geq 1 + \delta$.
\end{Case}
It then follows from~\eqref{eq:r_0_r_1_equivalent} that
$$\limsup_{p \to \infty} r_0^\ast \geq 1 + \delta,$$
and we obtain the following contradiction
\begin{equation}
\label{eq:r_0_ast_2}
1 = \int_S |x|^2 \omega_{\lambda, p}(x) dx \geq r_0(\omega_{\lambda, p})^2 \int_S \omega_{\lambda, p} \geq (r_0^\ast)^2 \int_S \omega_{\lambda, p} = (r_0^\ast)^2 > 1,
\end{equation}
with the constraints $M(\omega_{\lambda, p}) = L(\omega_{\lambda, p}) = 1$ for a number $p$ large enough.

\begin{Case}
$\limsup_{p \to \infty} r_1^\ast \leq 1 + \delta$.
\end{Case}
In this case, we have $ r_1^\ast \leq 1 + \delta$ for any number $p$ large enough. Combining the constraint $M(\omega_{\lambda, p}) = 1$ with~\eqref{eq:encadre_psi}, we get
\begin{equation}
\label{eq:integral_estimation_2}
\int_{r_1^\ast}^{2} \Big( \frac{\alpha_{\lambda, p}}{2} r^2 - \mu_{\lambda, p} - C_\lambda \Big)^{{p'} - 1} r dr \leq \frac{1}{\lambda \pi}.
\end{equation}
The computation of the integral leads to 
$$\frac{\alpha_{\lambda, p}}{2} 4 - \mu_{\lambda, p} - C_\lambda \leq \bigg( \frac{{p'} \alpha_{\lambda, p}}{\lambda \pi}\bigg)^\frac{1}{{p'}}.$$
Assuming that $\delta$ is small enough and applying the Young inequality to the right-hand side of this inequality give
$$\frac{\alpha_{\lambda, p}}{2} 4 - \mu_{\lambda, p} - C_\lambda \leq \frac{\delta\, \alpha_{\lambda, p}}{2} + \frac{1}{p} \bigg( \frac{2}{\delta \lambda \pi} \bigg)^\frac{p}{{p'}},$$
and then
$$(4 - \delta) \frac{\alpha_{\lambda, p}}{2} \leq \frac{1}{p} \bigg( \frac{2}{\delta \lambda \pi} \bigg)^\frac{p}{{p'}} + \mu_{\lambda, p} + C_\lambda.$$
In view of the definition of $r_0^\ast$, we are led to
$$(r_0^\ast)^2 \geq \frac{(4 - \delta) (\mu_{\lambda, p} -
C_\lambda)}{\mu_{\lambda, p} + C_\lambda + \frac{1}{p} \big( \frac{2}{\delta
\lambda \pi} \big)^{p - 1}} \underset{\mu_{\lambda, p} \to \infty}{\to} 4^2 - \delta.$$
Since $\delta$ is arbitrary small, we finally get a contradiction by arguing as for~\eqref{eq:r_0_ast_2}.

We conclude that assumption~\eqref{eq:mu-infini} does not hold, and we can argue similarly in order to prove that 
$\mu_{\lambda, p}$ does not converge towards $- \infty$ when $p \to \infty$. This shows that the numbers $\mu_{\lambda, p}$ remain bounded in this limit, and so do the numbers $\alpha_{\lambda, p}$ due to the constraint $M(\omega_{\lambda, p}) = 1$. \qed

\subsection{Proof of Lemma~\ref{lem:omega-bounded}}

The proof relies on standard regularity theory. In view of the definition of the function $\psi_{\lambda, p}$, we can derive from~\eqref{eq:Ks_omega_borne} and Lemma~\ref{lem:alpha_mu_bounded} that
\begin{equation}
\label{eq:psi-unif}
\| \psi_{\lambda, p} \|_{L^\infty} \leq C_\lambda,
\end{equation}
where $C_\lambda$ denotes, here as in the sequel, a positive number not depending on $p$ and possibly changing from line to line. In view of~\eqref{eq:max-prob-eq}, we obtain
$$\| \omega_{\lambda, p} \|_{L^\infty} \leq \lambda C_\lambda^\frac{1}{p - 1}.$$
Since $p \geq p_0$, this is enough to get
$$\| \omega_{\lambda, p} \|_{L^\infty} \leq C_\lambda,$$
and to deduce~\eqref{eq:bound-omega} from the constraint $M(\omega_{\lambda, p}) = 1$ and the H\"older inequality. As a consequence of~\eqref{eq:psi-unif}, we also observe that
\begin{equation}
\label{eq:psi-Lr}
\| \psi_{\lambda, p} \|_{L^r(B(0, R))} \leq C_\lambda R^\frac{2}{r},
\end{equation}
for any $1 \leq r \leq \infty$ and any positive number $R$.

For $1 < r < \infty$, we next recall the existence of a positive number $A_{s, r}$, depending only on $s$ and $r$, such that
\begin{equation}
\label{eq:2s-regularity}
\| K_s f \|_{\dot{W}^{2 s, r}(\RR^2)} \leq A_{s, r} \| f \|_{L^r(\RR^2)},
\end{equation}
for any function $f \in L^r(\RR^2)$. Applying this inequality to the function $f = \omega_{\lambda, p} \mathbbm{1}_{\supp(\omega_{\lambda, p})}$, we obtain
$$\| K_s \omega_{\lambda, p} \|_{\dot{W}^{2 s, r}(\RR^2)} \leq A_{s, r} \| \omega_{\lambda, p} \|_{L^r(S)},$$
and~\eqref{eq:bound-psi} results from~\eqref{eq:psi-Lr} and Lemma~\ref{lem:alpha_mu_bounded}.

When $1/2 \leq s < 1$, this guarantees that the function $\psi_{\lambda, p}$ lies in $W^{1, r}(S)$ for any $1 < r < \infty$, so that its positive part $(\psi_{\lambda, p})_+$ is also in these spaces. Hence, we can compute
$$\nabla^\perp \Big( (\psi_{\lambda, p})_+^{p'} \Big) = p' (\psi_{\lambda, p})_+^{p' - 1} \nabla^\perp \psi_{\lambda, p} = \frac{p'}{\lambda}\omega_{\lambda, p} \nabla^\perp \psi_{\lambda, p},$$
and this identity holds in the sense of integrable functions. Given a function $\varphi \in \cC_c^\infty(S)$, we conclude that
$$\int_S \omega_{\lambda, p}(x) \, \nabla^\perp \psi_{\lambda, p}(x) \cdot \nabla \varphi(x) \, dx = \int_S \nabla^\perp \Big( (\psi_{\lambda, p})_+^{p'} \Big)(x) \cdot \nabla \varphi(x) \, dx = 0,$$
by integrating by parts. \qed

\subsection{Proof of Lemma~\ref{lem:Riesz-bathtub}}

Invoking the Riesz rearrangement inequality, we reduce the analysis of the maximization problem~\eqref{eq:double_integral_to_maximize} to the situation where both the functions $g$ and $h$ are even and non-increasing on $\RR_+$. In this case, since the function $f$ is even and non-increasing on $\RR_+$, so is the function $f \star g$. In particular, the sets $\{ f \star g > t \}$ have finite measure for any positive number $t$. Therefore, we can apply the bathtub principle in~\cite[Theorem 1.14]{LiebLos0} in order to show that the function $h_\mu = \mathbbm{1}_{\big[ - \frac{\mu}{2}, \frac{\mu}{2} \big]}$ is a solution to the problem
$$\cJ_\mu := \max_{h \in \cG_\mu(\RR)} \int_\RR f \star g(y) \, h(y) \, dy.$$
Moreover, when $f$ is decreasing on $\RR_+$, so is the function $f \star g$, and the bathtub principle guarantees the uniqueness of the solution $h_\mu$. We finally complete the proof of Lemma~\ref{lem:Riesz-bathtub} by observing that the functions $g$ and $h$ play symmetric roles in the maximization problem~\eqref{eq:double_integral_to_maximize}. \qed

\subsection{Proof of Proposition~\ref{prop:patch-exist}}
\label{sub:prop-patch-exist}

Let $p > 1/s$. Consider a solution $\omega_{\lambda, p}$ to the maximization problem~\eqref{eq:pen-max-prob} and the corresponding Lagrange multipliers $\alpha_{\lambda, p}$ and $\mu_{\lambda, p}$ given by Lemma~\ref{lem:penalized-max}. Up to an omitted subsequence, we first derive from Lemma~\ref{lem:alpha_mu_bounded} the existence of two numbers $\alpha_\lambda$ and $\mu_\lambda$ such that
\begin{equation}
\label{eq:conv-alpha-mu}
\alpha_{\lambda, p} \to \alpha_\lambda \quad \text{and} \quad \mu_{\lambda, p} \to \mu_\lambda,
\end{equation}
when $p \to \infty$. In view of Lemma~\ref{lem:omega-bounded}, we can invoke the Sobolev embedding theorem in order to exhibit a continuous function $\psi_\lambda : \RR^2 \to \RR$ such that, up to a further subsequence,
\begin{equation}
\label{eq:conv-psi}
\psi_{\lambda, p} \to \psi_\lambda \text{ in } L^\infty(B(0, R)),
\end{equation}
as $p \to \infty$, for any positive number $R$. In particular, this convergence provides
\begin{equation}
\label{eq:smaller_than_lambda}
\limsup_{p \to \infty} \big\| \omega_{\lambda, p} \big\|_{L^\infty(S)} \leq \lambda \lim_{p \to \infty} \big\| (\psi_{\lambda, p})_+ \big\|_{L^\infty(S)}^{p' - 1} \leq \lambda.
\end{equation}
Going back to Lemma~\ref{lem:omega-bounded}, we can apply the Banach-Alaoglu theorem to construct a non-negative function $\omega_\lambda \in L^\infty(S)$ such that, up to a further subsequence,
\begin{equation}
\label{eq:conv-omega}
\omega_{\lambda, p} \overset{\star}{\rightharpoonup} \omega_\lambda \text{ in } L^\infty(S),
\end{equation}
when $p \to \infty$. In view of~\eqref{eq:smaller_than_lambda}, the function $ \omega_\lambda$ lies in $L_\lambda^\infty(S)$ and we can extend it as a function in $X_\lambda$ by performing a $N$-fold symmetrization. Moreover, we infer from~\eqref{eq:conv-alpha-mu},~\eqref{eq:conv-psi} and~\eqref{eq:conv-omega} that the function $\psi_\lambda$ is given by~\eqref{def:psi-lambda-S}.

We now check that this function solves the maximization problem~\eqref{def:max-prob-S}. Since the maps $x \mapsto 1$ and $x \mapsto |x|^2$ are locally integrable, we derive from the constraints $M(\omega_{\lambda, p}) = L(\omega_{\lambda, p}) = 1$ and~\eqref{eq:conv-omega} that $M(\omega_\lambda) = L(\omega_\lambda) = N$. Consider next a function $\omega \in X_\lambda$ such that $M(\omega) = L(\omega) = N$. It follows from Lemma~\ref{lem:penalized-max} and~\eqref{eq:smaller_than_lambda} that
$$E_s(\omega) \leq E_s(\omega_{\lambda, p}) + \frac{\lambda N}{p} \int_S \Big( \frac{\omega}{\lambda} \Big)^p \leq E_s(\omega_{\lambda, p}) + \frac{\lambda N |S|}{p}.$$
Invoking the compact embedding of $L^\infty(S)$ into $\dot{H}^{- s}(S)$, we can extract a further subsequence for which
$$E_s(\omega_{\lambda, p}) \to E_s(\omega_\lambda),$$
when $p \to \infty$. Hence, we are led to
$$E_s(\omega) \leq E_s(\omega_\lambda),$$
by taking the limit $p \to \infty$ in the previous inequality. Therefore, the function $\omega_\lambda$ solves the maximization problem~\eqref{def:max-prob-S}.

At this point, we explain why the function $\omega_\lambda$ is a vortex patch. Given any function $\omega \in L_\lambda^\infty(S)$, we define the function $\widehat{\omega}$ as being the only angular Steiner symmetric function of $L_\lambda^\infty(S)$, with values into the pair $\{ 0, \lambda \}$, such that
$$\int_{- \frac{\pi}{2N}}^\frac{\pi}{2N} \omega(r, \theta) \, d\theta = \int_{- \frac{\pi}{2N}}^\frac{\pi}{2N} \widehat{\omega}(r, \theta) \, d\theta,$$
for almost any $\frac{1}{2} < r < 2$. Using the decreasing nature of the map $\xi
\mapsto \kappa_s(r, r', \xi)$ given by Lemma~\ref{lem:combinatorics}, we deduce from Lemma~\ref{lem:Riesz-bathtub} that
\begin{align*}
\int_{- \frac{\pi}{2 N}}^\frac{\pi}{2 N} \int_{- \frac{\pi}{2 N}}^\frac{\pi}{2 N} k_s(r, r', \theta - \theta') \, \omega_1(r, \theta) \, \omega_2(r', \theta') & \, d\theta \, d\theta'\\
& \leq \int_{- \frac{\pi}{2 N}}^\frac{\pi}{2 N} \int_{- \frac{\pi}{2 N}}^\frac{\pi}{2 N} k_s(r, r', \theta - \theta') \, \widehat{\omega_1}(r, \theta) \, \widehat{\omega_2}(r', \theta') \, d\theta \, d\theta',
\end{align*}
for any functions $(\omega_1, \omega_2) \in L_\lambda^\infty(S)^2$ and almost
any $(r, r') \in (\RR_+)^2$. Moreover, this inequality is an equality if and
only if $\omega_1(r, \cdot) = \widehat{\omega_1}(r, \cdot) $ and $\omega_2(r,'
\cdot) = \widehat{\omega_2}(r', \cdot) $ for almost any $- \pi/(2 N) \leq \theta
\leq \pi/(2 N)$. Therefore, we deduce
$$E_s(\omega) \leq E_s(\widehat{\omega}),$$
with equality if and only if $\omega = \widehat{\omega}$. This guarantees that the function $\omega_\lambda = \widehat{\omega_\lambda}$ is angular Steiner symmetric, and up to the multiplicative factor $\lambda$, equal to the characteristic function of a measurable subset $\Omega_\lambda$ of $S$.

Since $\omega_{\lambda, p} = (\psi_{\lambda, p})_+^{p' - 1}$, we derive from~\eqref{eq:conv-psi} that this set satisfies
$$\big\{ \psi_\lambda > 0 \big\} \subseteq \Omega_\lambda \subseteq \big\{ \psi_\lambda \geq 0 \big\}.$$
In order to conclude that this set matches (up to a set of measure $0$) with the
support of the function $(\psi_\lambda)_+$, we are reduced to show that the set
$\{ \psi_\lambda = 0 \}$ has zero measure. Let $\frac{1}{2} < r < 2$ and $0 < \theta_1, \theta_2 < \pi/4$, with $\theta_1 < \theta_2$. In view of definition~\eqref{def:psi-lambda-S}, we compute
\begin{align*}
& \psi_\lambda(r, \theta_1) - \psi_\lambda(r, \theta_2) = K_s \omega_\lambda(r, \theta_1) - K_s \omega_\lambda(r, \theta_2)\\
& = \int_{\frac{1}{2}}^{2} \int_0^\frac{\pi}{2 N} \Big( k_s(r, r', \theta' - \theta_1) - k_s(r, r', \theta' - \theta_2) \Big) \, \omega_\lambda(r', \theta') \, r' \, dr' \, d\theta,
\end{align*}
and this double integral is positive due to the fact that the map $\xi \mapsto k_s(r, r', \xi)$ is decreasing on $(0, \pi/2 N)$ as a consequence of Lemma~\ref{lem:combinatorics} below. By the Fubini theorem, we conclude that the measure of the set $\{ \psi_\lambda = 0 \}$ is equal to $0$, which eventually provides~\eqref{eq:omega-lambda}.

We finally turn to the proof of~\eqref{eq:weak-formulation-S}. When $1/2 \leq s < 1$, Lemma~\ref{lem:omega-bounded} guarantees that the functions $\psi_{\lambda, p}$ are uniformly bounded in $W^{1, r}(B(0, R))$ with respect to $p \to \infty$ for any exponent $1 < r < \infty$ and any positive number $R$. As a consequence, we can assume that the function $\psi_\lambda$ lies in $W_{\text{loc}}^{1, r}(\RR^2)$. In particular, we are allowed to write
$$\nabla^\perp \big( (\psi_{\lambda, p})_+ \big) = \mathbbm{1}_{\{ \psi_\lambda > 0 \}} \nabla^\perp \psi_{\lambda, p},$$
so that, by integrating by parts,
$$\int_S \omega_\lambda(x) \, \nabla^\perp \psi_\lambda(x) \cdot \nabla \varphi(x) \, dx = \lambda \int_S \nabla^\perp \big( \psi_\lambda \big)_+(x) \cdot \nabla \varphi(x) \, dx = 0,$$
for any function $\varphi \in \cC_c^\infty(S)$. This completes the proof of Proposition~\ref{prop:patch-exist}. \qed

\subsection{Proof of Lemma~\ref{lem:energy}}

We infer from the Riesz rearrangement inequality that
\begin{equation*}
\int_X \int_X \left( \tfrac{\varepsilon}{|x - x'|} \right)^{2 (1 - s)} \omega(x) \, \omega(x') \, dx \, dx' 
\leq 
\lambda^2 \int_{\cB(0, \varepsilon \sqrt{M(\omega, X)})} \int_{\cB(0,
\varepsilon \sqrt{M(\omega, X)})} \left( \tfrac{\varepsilon}{|x-x'|} \right)^{2 (1 - s)} \, dx \, dx',
\end{equation*}
and the first inequality in the statement of the lemma follows from the definition of $\cI_s.$
Regarding the second statement, 
since the map $x \mapsto |x|^{2 (s - 1)}$ is radially decreasing, we can rearrange the function $\omega$ so as to write
\begin{align*}
\int_X \omega(x') \int_{X'} \bigg( \frac{\varepsilon}{|x - x'|} \bigg)^{2 (1 - s)} \omega(x) \, dx \, dx' 
& \leq \lambda \int_X \omega(x') \int_{\cB(x, \varepsilon \sqrt{M(\omega, X')})} 
\bigg( \frac{\varepsilon}{|x - x'|} \bigg)^{2 (1 - s)} \, dx \, dx'\\
& = \lambda \, M(\omega, X) \int_{\cB(0, \varepsilon \sqrt{M(\omega, X')})} 
\bigg( \frac{\varepsilon}{|y|} \bigg)^{2 (1 - s)} \, dy.
\end{align*}
and the conclusion follows from the identity $\lambda \pi \varepsilon^2 = 1$.
Finally, the last statement is a direct consequence of the inequality
$$|x - x'| \geq d(X, X'),$$
which holds for any $x \in X$ and $x' \in X'$.\qed

\subsection{Proof of Lemma~\ref{lem:E-s-bounds}}

On the one hand, since $\omega_\lambda$ is the maximizer of the energy $E_s$,
\begin{equation}
\label{eq:energy bigger that the ball}
E_s(\varpi_\lambda) \leq E_s(\omega_\lambda),
\end{equation}
where $\varpi_\lambda$ was defined in~\eqref{eq:overline omega}. Then,
$$E_s(\varpi_\lambda) \geq c_s \int_S \int_S \frac{\lambda
\mathbbm{1}_{\cB((r_\varepsilon, 0), \varepsilon)}(x) \,
\lambda\mathbbm{1}_{\cB((r_\varepsilon, 0), \varepsilon)}(x')}{|x - x'|^{2 (1 -
s)}} \, dx \, dx' = \frac{c_s}{\varepsilon^{2 (1 - s)}}\cI_s.$$
Together with~\eqref{eq:energy bigger that the ball}, this gives the lower bound in~\eqref{eq:rescaled energy}. Concerning the upper bound,
$$E_s(\omega_\lambda) \leq c_s \int_S \int_S \frac{\omega_\lambda(x) \omega_\lambda(x')}{|x - x'|^{2 (1 - s)}} \, dx \, dx' + C.$$
Using the Riesz rearrangement inequality,
$$E_s(\omega_\lambda) \leq c_s \int_S \int_S \frac{\lambda \mathbbm{1}_{\cB(0,
\varepsilon)}(x) \, \lambda \mathbbm{1}_{\cB(0, \varepsilon)}(x')}{|x - x'|^{2
(1 - s)}} \, dx \, dx' + C = \frac{c_s}{\varepsilon^{2 (1 - s)}} \cI_s + C.$$
The lemma is proved. \qed

\subsection{Proof of Lemma~\ref{lem:Iam not small}}

First, we expand the double integral in~\eqref{eq:rescaled energy} as
\begin{align*}
I_s(\omega,S,S) = & \int_S \int_{\cB(x, \Lambda_0 \varepsilon)} \bigg( \frac{\varepsilon}{|x - x'|}\bigg)^{2 (1 - s)} 
\omega(x')dx'\omega(x)dx\\
& + \int_S \int_{\cB(x, \Lambda_0 \varepsilon)^c} 
\bigg( \frac{\varepsilon}{|x - x'|}\bigg)^{2 (1 - s)} 
\omega(x')dx' \, \omega(x)dx.
\end{align*}
Let $\eta > 0$ and $C > 0$ to be precised later, and set
\begin{equation}
\label{eq:set Omega}
\Upsilon_\eta^{\Lambda_0} := \Big\{ x \in \supp(\omega) : \meas \big( \cB(x, \Lambda_0 \varepsilon) 
\cap \supp(\omega)\big) \geq \eta \pi \varepsilon^2 \Big\}.
\end{equation}
Assume for the sake of a contradiction that $\Upsilon_\eta^{\Lambda_0}$ is empty. Then a rearrangement argument 
with respect to the variable $x'$ yields
$$
\int_S \int_{\cB(x, \Lambda_0 \varepsilon)} \bigg( \frac{\varepsilon}{|x - x'|}
\bigg)^{2 (1 - s)} \omega(x')dx' \omega(x)dx 
< 
\lambda \int_S \int_{\cB(x, \sqrt{\eta} \varepsilon)} 
\bigg( \frac{\varepsilon}{|x - x'|}\bigg)^{2 (1 - s)} dx' \omega(x) dx = \frac{\eta^s}{s}.$$
On the other hand, we check that
$$
\int_S \int_{\cB(x, \Lambda_0 \varepsilon)^c} 
\bigg( \frac{\varepsilon}{|x - x'|} \bigg)^{2 (1 - s)}
\omega(x')dx'\omega(x)dx \leq \frac{1}{\Lambda_0^{2 (1 - s)}},
$$
so that using~\eqref{eq:rescaled energy} we obtain 
$$
s\cI_s < \eta^s + \frac{s}{\Lambda_0^{2 (1 - s)}}.
$$
In view of the inequality $s\cI_s \geq \frac16,$ this provides the desired 
contradiction if we choose $\eta$ sufficiently small and $\Lambda_0$ 
sufficiently large (both depending only on $s$). It suffices then to choose 
for $x$ an arbitrary point in $\Upsilon_{\eta}^{\Lambda_0}.$
\qed

\subsection{Proof of Lemma~\ref{lem:decay}}

\indent\emph{Step 1.} Let $A \subset B$ be arbitrary measurable subsets in $\RR^2$ such
that $K\varepsilon:= d(A, B^c) > 0.$ We decompose the double integral in Lemma~\ref{lem:E-s-bounds} as
\begin{equation}
\label{eq:six big integrals}
\begin{split}
\cI_s &\leq I_s(\omega, A, A) + I_s(\omega, B\setminus A, B\setminus A) +
I_s(\omega, B^c, B^c)\\ 
&+ 2 I_s(\omega, A\cup B^c, B \setminus A) + 2I_s(\omega, A,
B^c).
\end{split}
\end{equation}
To further simplify the notations, we also set
\begin{align*}
\eta_0 := M(\omega, A), \ \eta_1 := M(\omega, B \setminus A) \text{ and }
\eta_2 := M(\omega, B^c),
\end{align*}
and in particular since $M(\omega) = 1$ we have $\eta_0 + \eta_1 + \eta_2 = 1.$

We now apply the estimates in Lemma~\ref{lem:energy} to~\eqref{eq:six big
integrals}. More precisely, we apply the first one for the first three terms in~\eqref{eq:six big
integrals}, the second one for the fourth term, and the third one for the last
term. This yields, after division by $\cI_s$, 
\begin{equation}
\label{eq:beautiful estimate}
1 \leq \eta_0^{1 + s} + \eta_1^{1 + s} + \eta_2^{1 + s} 
+ \frac{2}{\cI_s} \bigg( \frac{\eta_1 (\eta_0 + \eta_2)^s}{s} 
+ \frac{\eta_0 \eta_2}{K^{2(1 - s)}} \bigg).
\end{equation}
Since $\eta_0$, $\eta_1$, $\eta_2$ and $\eta_0 + \eta_2$ are less than $1$, we
further deduce 
\begin{equation}
\label{eq:smaller than one}
1 \leq \eta_0^{s + 1} + \eta_2^{s + 1} + \Big( 1 + \frac{2}{s \cI_s} \Big) \eta_1 + \frac{2}{\cI_s K^{2 (1 - s)}}.
\end{equation}
In~\eqref{eq:smaller than one}, we wish to view the last two terms as negligible, which
requires in particular $\eta_1$ to be small, and then use a convexity
inequality to deduce that the one on the left hand side cannot be too much spread across
$\eta_0$ and $\eta_2.$ More precisely, we require $K$ to satisfy $K \geq
\Lambda_0$, where $\Lambda_0$ is the constant obtained from Lemma 
\ref{lem:Iam not small}, and we define $A = \cB(x, mK\varepsilon)$ and $B = \cB(x,
(m+1)K\varepsilon)$ where $x$ is also given by Lemma~\ref{lem:Iam not small} and 
$m \in \mathbb{N}_\ast$ is to be chosen hereafter. Note that since $\cB(x,
\Lambda_0 \varepsilon) \subset A$, by Lemma~\ref{lem:Iam not small} we have
\begin{equation}
\label{eq:eta0 pas small}
\eta_0 \geq \eta > 0.
\end{equation}
The integer $m \geq 1$ is chosen in such a way that $\eta_1 = M(\omega,
B\setminus A)$ is small. By additivity of the integral, and since
$M(\omega_\lambda, S) = 1,$ we may find $1 \leq m \leq K^{2(1-s)}$
such that $\eta_1 \leq 1 / K^{2(1-s)}.$ From~\eqref{eq:smaller than one}, we
therefore obtain 
$$
1 - \eta_0^{s + 1} - (1 - \eta_0)^{s + 1} \leq 
\Big( 1 + \frac{2 (1 + s)}{s \cI_s} \Big) \frac{1}{K^{2(1 - s)}}.
$$
We now use the concavity of the function $t \longmapsto 1 - t^{s + 1} - (1 - t)^{s + 1}$ 
on the segment $[0, 1]$ in order to obtain
$$1 - t^{s + 1} - (1 - t)^{s + 1} \geq (1 - 2^{- s}) \bigg( \frac{1}{2} - \Big| t - \frac{1}{2} \Big| \bigg),$$
so that we are reduced to the alternative
\begin{equation}
\label{eq:alternative}
\eta_0 \leq \frac{C}{K^{2(1 - s)}} \quad \text{or} \quad 1 - \eta_0 \leq
\frac{C}{K^{2(1 - s)}},
\end{equation}
for some constant $C$ depending only on $s.$ If $K$ is chosen sufficiently
large, only the first alternative can hold in view of~\eqref{eq:eta0 pas small}.
Note also that $mK \leq K ^{1 + 2(1-s)}$. We finally fix the value of $K$ in such
a way that $K^{1 + 2(1-s)} = \Lambda$, provided $\Lambda$ is sufficiently large so that the
previous requirements on $K$ hold with such a choice. If it is not the case, then the 
conclusion of Lemma~\ref{lem:decay} can simply be obtained by choosing $C$ 
sufficiently large so that $C/\Lambda^{\gamma_s}\geq 1$, in which case it yields 
a trivial inequality. Therefore, 
\begin{equation}
\label{eq:tagada}
M(\omega, \cB(x,\Lambda\varepsilon)^c) \leq 1 - \eta_0 \leq \frac{C}{K^{2(1-s)}} =
\frac{C}{\Lambda^{\gamma_s}}.
\end{equation}

\emph{Step 2.} To conclude the proof this lemma, we show that in the above 
estimate, $x$ can be chosen equal to $(1,0)$. We start with the following lower 
and upper bound of $L(\omega)$ using $|x'|\leq 2$,
\begin{equation}
\label{eq:fatigue}
\begin{split}
\big( \inf_{x' \in \cB(x, \Lambda\varepsilon)} &|x'|^2\big) M\big(\omega,\cB(x, \Lambda \varepsilon)\big) \leq L(\omega) \\&\leq \big( \sup_{x' \in \cB(x, \Lambda \varepsilon)} |x'|^2 \big) \, 
M\big(\omega,\cB(x, \Lambda \varepsilon)\big) + 4 M\big(\omega,S\setminus\cB(x, \Lambda \varepsilon)\big).
\end{split}
\end{equation}
We note that in the statement of Lemma~\ref{lem:Iam not small}, if in addition $\omega$ is assumed to be 
angular Steiner symmetric then the point $x$ can be chosen of the form 
$x = (r, 0)$. Indeed, under Steiner symmetry the required 
inequality is improved
by shifting $x$ along the angular variable until it reaches the horizontal
axis. With $x=(r,0)$, the estimate~\eqref{eq:fatigue} becomes
\begin{align*}
\Big( r - \frac{\Lambda^2}{\pi \lambda} \Big) M\big(\omega,\cB(x, \Lambda \varepsilon)\big) \leq L(\omega) \leq \Big( r + \frac{\Lambda^2}{\pi \lambda} \Big) \,M\big(\omega,\cB(x, \Lambda \varepsilon)\big) + 4 M\big(\omega,S\setminus\cB(x, \Lambda \varepsilon)\big).
\end{align*}
In view of the constraints $M(\omega) = L(\omega) = 1$ and using~\eqref{eq:tagada},
$$\Big (r- \frac{\Lambda^2}{\pi \lambda} \Big) \, \Big( 1 - 
\frac{C}{\Lambda^{\gamma_s}} \Big) \leq 1 \leq r+ \frac{\Lambda^2}{\pi \lambda} + 4\frac{C}{\Lambda^{\gamma_s}}.$$
Since $\Lambda$ is any positive number, in the limit $\lambda \to \infty$, this gives $x=(r,0)\to(1,0)$. 
Using again the constraint $M(\omega) = L(\omega) = 1$ and ~\eqref{eq:tagada} improves the convergence with $|x-(1,0)|\leq C\varepsilon.$ With such a convergence rate at hand, we can transform~\eqref{eq:tagada} into
$$
M(\omega, \cB((1,0),\Lambda\varepsilon)^c) \leq \frac{C}{\Lambda^{\gamma_s}},
$$
which is the announced estimate.
\qed

\subsection{Proof of Lemma~\ref{lem:jen ai marre}}

In this proof, we focus on the support of the vorticity, that is the subset of 
$S$, which is composed of the points $x \in S$ such that $\psi_\lambda(x) \geq 
0$. Under this condition, we can come back to the definition of the function 
$\psi_\lambda$ in order to obtain the inequality
\begin{equation}
\label{eq:the difference}
0 \geq \psi_\lambda(y) - \psi_\lambda(x) = K_s \omega_\lambda(y) - 
K_s \omega_\lambda(x) + \frac{\alpha}{2} \big( |y|^2 - |x|^2 \big), 
\end{equation}
for any point $y \in S$ such that $\psi_\lambda(y) \leq 0$. In 
order to establish that the support of the vorticity has no intersection with 
the complementary set of $\cB( (1,0), 2 \Lambda \varepsilon) \cup B$, we rely on 
the previous inequality for two different choices of the point $y$. Technically, 
our argument first requires to prove the two following estimates on the function 
$K_s \omega_\lambda$.

Let $\Lambda$ be a positive number. When $z \in \cB( (1, 0), \Lambda 
\varepsilon)$, we can use the fact that $\cB( (1,0), \Lambda \varepsilon) 
\subseteq S$ in order to get the lower estimate
\begin{equation} 
\label{eq:estimate K_s below}
K_s \omega_\lambda(z) \geq c_s \int_S\frac{\omega_\lambda(y)}{|z - y|^{2(1 - s)}} \, dy - C \geq \frac{c_s}{(2 
\Lambda \varepsilon)^{2 (1 - s)}} M(\omega_\lambda, \cB( (1, 0), \Lambda 
\varepsilon)) - C.
\end{equation}
On the contrary, for $z \in \cB( (1, 0), 2 
\Lambda \varepsilon)^c$, we can split the integration domain $S$ into the ball 
$\cB( (1, 0), \Lambda \varepsilon)$ and its complementary set, so as to obtain 
$$K_s \omega_\lambda(z) \leq c_s \int_S \frac{\omega_\lambda(y)}{|z - y|^{2 (1 - 
s)}} \, dy + C \leq c_s \int_{S \setminus \cB( (1, 0), \Lambda \varepsilon)} 
\frac{\omega_\lambda(y)}{|z - y|^{2 (1 - s)}} \, dy + \frac{C_s}{(R 
\varepsilon)^{2 (1 - s)}} + C.$$ A rearrangement argument then provides the 
upper estimate
\begin{equation}
\label{eq:estimate K_s above}
K_s \omega_\lambda(z) \leq \frac{c_s}{s \varepsilon^{2 (1 - s)}} M(\omega_\lambda, S 
\setminus \cB( (1, 0), \Lambda \varepsilon))^s + \frac{c_s}{(\Lambda 
\varepsilon)^{2 (1 - s)}} + C.
\end{equation}

With these estimates at hand, we now assume for the sake of a contradiction the 
existence of a point $x \in S \setminus \big( \cB( (1, 0), 2 \Lambda 
\varepsilon) \cup B \big)$ such that $\psi_\lambda(x) \geq 0$. Recall here that 
\begin{equation}
\label{def:B again}
B = \Big\{ (r, \theta) \in S : \frac{1}{2} 
\leq r \leq \frac{1}{2} + \frac{1}{8} \text{ or } 2 - \frac{1}{8} \leq r \leq 2 
\Big\}.
\end{equation}
When $\varepsilon$ is small enough, we can find a point 
$y_1 \in S$ such that $\psi_\lambda(y_1) \leq 0$, and which satisfies the two 
conditions
\begin{equation}
\label{eq:pouet pouet}
\sign(\alpha) = \sign \big( 
|y_1|^2 - |x|^2 \big) \quad \text{and} \quad \big| |y_1|^2 - |x|^2 \big| \geq 
\frac{1}{10}.
\end{equation}
Such a point necessarily exists due to the 
constraints $M(\omega_\lambda) = L(\omega_\lambda) = 1$, the vorticity 
concentration provided by Lemma~\ref{lem:decay} and the fact that $x\in 
S\setminus B$. Combining the inequalities~\eqref{eq:the difference} 
and~\eqref{eq:pouet pouet} with the estimate~\eqref{eq:estimate K_s above} leads 
to the bound
\begin{equation}
\label{eq:estimate alpha}
\frac{|\alpha|}{20} \leq 
\frac{c_s}{s \varepsilon^{2 (1 - s)}} M(\omega_\lambda, S \setminus \cB( (1, 0), 
\Lambda \varepsilon) )^s + \frac{c_s}{(\Lambda \varepsilon)^{2 (1 - s)}} + C. 
\end{equation}

Consider now a further positive number $\Lambda_1$, as well as a point $y_2 \in 
\cB( (1,0), \Lambda_1 \varepsilon)$ such that $\psi_\lambda(y_2) \leq 0$. Provided that $\Lambda_1 > 1$, the 
existence of this point follows from the property that the support of the 
vorticity has measure $\pi \varepsilon^2$. In 
view of~\eqref{eq:estimate K_s below} and~\eqref{eq:estimate K_s above}, we can 
estimate~\eqref{eq:the difference} so as to obtain
$$\frac{c_s}{(2 \Lambda_1 \varepsilon)^{2 (1 - s)}} M(\omega_\lambda, \cB( (1, 0), 
\Lambda_1 \varepsilon)) \leq \frac{c_s}{s \varepsilon^{2 (1 - s)}} 
M(\omega_\lambda, S \setminus \cB( (1, 0), \Lambda \varepsilon))^s + 
\frac{c_s}{(\Lambda \varepsilon)^{2 (1 - s)}} + 4 |\alpha| + C.$$
Hence, we deduce from~\eqref{eq:estimate alpha} that
$$\frac{1}{(2 \Lambda_1)^{2 
(1 - s)}} M(\omega_\lambda, \cB( (1, 0), \Lambda_1 \varepsilon)) \leq C \Big( 
M(\omega_\lambda, \cB((1, 0), S \setminus \Lambda \varepsilon))^s + 
\frac{1}{\Lambda^{2 (1 - s)}} + \varepsilon^{2 (1 - s)}\Big).$$
We finally 
invoke the decay estimate provided by Lemma~\ref{lem:decay} in order to get 
$$\frac{1}{(2 \Lambda_1)^{2 (1 - s)}} \Big( 1 - \frac{1}{\Lambda_1^{\gamma_s}} 
\Big) \leq C\Big( \frac{1}{\Lambda^{s\gamma_s}} + \frac{1}{\Lambda^{2 (1 - s)}} 
+ \varepsilon^{2 (1 - s)}\Big).$$
For $\varepsilon$ small enough, this 
inequality is false for a fixed number $\Lambda_1$, when $\Lambda$ is chosen 
large enough. This contradiction concludes the proof of Lemma~\ref{lem:jen ai 
marre}. \qed

\subsection{Proof of Proposition~\ref{prop:condsupp}}

We argue by contradiction. We assume that the function $\omega_\lambda$ does not 
identically vanish on the set $B$ and we derive a contradiction with the 
property that it is a solution of the maximization 
problem~\eqref{def:max-prob-S}. More precisely, our main argument lies in 
constructing a function $\omega_{\lambda, 2} \in X_\lambda$, with 
$$M(\omega_{\lambda, 2}) = L(\omega_{\lambda, 2}) = 1,$$
and such that 
$$E_s(\omega_{\lambda, 2}) > E_s(\omega_\lambda).$$
This construction is 
performed in two steps. The first step provides a function $\omega_{\lambda, 
1}$, which satisfies all the previous criteria except the constraint 
$L(\omega_{\lambda, 2}) = 1$. This constraint is recovered in a second step 
without losing too much energy.

\emph{Step 1.} We assume for the sake of a contradiction that $M(\omega_\lambda, 
B) > 0$. The construction of the function $\omega_{\lambda, 1}$ then relies on 
Lemma~\ref{lem:jen ai marre}, which guarantees that the function 
$\omega_\lambda$ identically vanishes on the subset $S \setminus (\cB( (1, 0), 
\Lambda \varepsilon) \cup B)$. The idea underlying the construction is to 
rearrange the function $\omega_\lambda$ so that the positive mass 
$M(\omega_\lambda, B)$ localized in $B$ is brought back into a small disk nearby 
$\cB( (1, 0),\Lambda \varepsilon)$. 

Observe first that the localized mass $M(\omega_\lambda, B)$ can be estimated as 
\begin{equation}
\label{eq:mass-residual}
M(\omega_\lambda, B) \leq C 
\varepsilon^{\gamma_s},
\end{equation}
by Lemma~\ref{lem:decay}. In particular, 
this quantity vanishes in the limit $\varepsilon \to 0$. Setting 
\begin{equation}
\label{def:rho1} \rho_1 = \sqrt{M(\omega_\lambda, B)} \quad 
\text{and} \quad x_1 = 1 + \Lambda \varepsilon + \rho_1 \varepsilon, 
\end{equation}
we check that the ball $\cB(x_1, \rho_1 \varepsilon)$ lies in 
the subset $S \setminus (\cB( (1, 0), \Lambda \varepsilon) \cup B)$ for 
$\varepsilon$ small enough, and that its intersection with the ball $\cB( (1, 
0), \Lambda \varepsilon)$ reduces to the point $(1 + \Lambda \varepsilon, 0)$.

With these properties at hand, we can define the restriction of the function 
$\omega_{\lambda, 1}$ to the angular sector $\mfS_N$ as
$$\omega_{\lambda, 1} = \omega_\lambda \mathbbm{1}_{\cB( (1, 0), \Lambda \varepsilon)} + \lambda 
\mathbbm{1}_{\cB(x_1, \rho_1 \varepsilon)}.$$
We next extend this definition to 
the whole space using the $N$-fold symmetry. In view of this definition, the 
function $\omega_{\lambda, 1}$ belongs to the function set $X_\lambda$. 
Moreover, it is angular Steiner symmetric due to the angular Steiner symmetry of 
the function $\omega_\lambda$, and it satisfies the condition 
$M(\omega_{\lambda,1}) = M(\omega_\lambda) = 1$.

We now estimate how the energy of the function $\omega_\lambda$ has been 
modified by this construction. Recall first that the function $\omega_{\lambda, 
1}$ is equal to the function $\omega_\lambda$ on the ball $\cB( (1, 0), \Lambda 
\varepsilon)$, so that
\begin{equation}
\label{eq:energy variation 2} I_s \big( 
\omega_{\lambda, 1}, \cB( (1, 0), \Lambda \varepsilon), \cB( (1, 0), \Lambda 
\varepsilon) \big) = I_s \big( \omega_\lambda, \cB( (1, 0), \Lambda 
\varepsilon), \cB( (1, 0), \Lambda \varepsilon) \big).
\end{equation}
We next 
invoke the Riesz rearrangement inequality~\cite{Riesz1} in order to obtain 
\begin{equation}
\label{eq:energy variation 1} I_s \big( \omega_\lambda, S 
\setminus \cB( (1, 0), \Lambda \varepsilon), S \setminus \cB( (1, 0), \Lambda 
\varepsilon) \big) \leq I_s \big( \omega_{\lambda, 1}, S \setminus \cB( (1, 0), 
\Lambda \varepsilon), S \setminus \cB( (1, 0), \Lambda \varepsilon) \big). 
\end{equation}
We also check that
\begin{equation}
\label{eq:large-dist}
\text{dist} \big( \text{supp}(\omega_\lambda 
\mathbbm{1}_S), \text{supp}(\omega_\lambda \mathbbm{1}_{\RR^2 \setminus S}) 
\big) = \sin \Big( \frac{\pi}{N} \Big) \geq \frac{1}{2 N},
\end{equation}
which gives 
\begin{equation}
\label{eq:energy variation 3} \big| I_s(\omega_\lambda, S, 
\RR^2 \setminus S) - I_s(\omega_{\lambda, 1}, S, \RR^2 \setminus S) \big| \leq C 
\varepsilon^{2 (1 - s)} M(\omega_\lambda, B).
\end{equation}

Finally, we turn to the last difference, namely $I_s \big( \omega_{\lambda, 1}, 
\cB( (1, 0), \Lambda \varepsilon), S \setminus \cB( (1, 0),\Lambda \varepsilon) 
\big) - I_s \big( \omega_\lambda, \linebreak[0] \cB( (1, 0), \Lambda 
\varepsilon), S \setminus \cB( (1, 0),\Lambda \varepsilon) \big)$. We observe 
that
$$\text{dist} \big( \cB( (1, 0), \Lambda \varepsilon), B \big) \geq 
\frac{1}{4},$$
when $\varepsilon$ is small enough. Since the function 
$\omega_\lambda$ identically vanishes on $S \setminus (\cB( (1, 0), \Lambda 
\varepsilon) \cup B)$, it follows that
\begin{equation}
\label{eq:energy variation 4}
I_s \big( \omega_\lambda, \cB( (1, 0), \Lambda \varepsilon), S 
\setminus \cB( (1, 0), \Lambda \varepsilon) \big) \leq C \varepsilon^{2 (1 - s)} 
M(\omega_\lambda, B).
\end{equation}
On the other hand, we check that $$\max_{x 
\in \cB(x_1, \rho_1)} \, \max_{y \in \cB( (1, 0), \Lambda \varepsilon)} |x - y| 
= 2 \varepsilon \big( \Lambda + \rho_1),$$
and this inequality gives 
\begin{align*}
I_s \big( \omega_{\lambda, 1}, \cB( (1, 0), \Lambda \varepsilon), 
& S \setminus \cB( (1, 0), \Lambda \varepsilon) \big)\\ & \geq c_s \big( 2 
\Lambda + 2 \rho_1 \big)^{2 (s - 1)} M \big( \omega_{\lambda, 1}, S \setminus 
\cB( (1, 0), \Lambda \varepsilon) \big) \, M \big( \omega_{\lambda, 1}, \cB( (1, 
0), \Lambda \varepsilon) \big).
\end{align*}
In view of the decay estimate in 
Lemma~\ref{lem:decay}, we know that
$$M \big( \omega_{\lambda, 1}, \cB( (1, 0), 
\Lambda \varepsilon) \big) \geq \frac{1}{2},$$
for $\varepsilon$ small enough, 
while $M \big( \omega_{\lambda, 1}, S \setminus \cB( (1, 0), \Lambda \varepsilon) 
\big) = M(\omega_\lambda, B)$. As a consequence, we are led to the lower bound 
\begin{equation}
\label{eq:energy variation 5}
I_s \big( \omega_{\lambda, 1}, 
\cB( (1, 0), \Lambda \varepsilon), S \setminus \cB( (1, 0), \Lambda \varepsilon) 
\big) \geq C M(\omega_\lambda, B),
\end{equation}
and we can collect the energy 
estimates from~\eqref{eq:energy variation 2} to~\eqref{eq:energy variation 5} in 
order to write
\begin{equation}
\label{eq:ne pourrons nous jamais} 
E_s(\omega_{\lambda, 1}) - E_s(\omega_\lambda) = \frac{N}{2 \varepsilon^{2 (1 - 
s)}} \big( I_s(\omega_{\lambda, 1}, S, \RR^2) - I_s(\omega_\lambda, S, \RR^2) 
\big) \geq C \Big( \frac{1}{\varepsilon^{2 (1 - s)}} - 1 \Big) \, 
M(\omega_\lambda, B).
\end{equation}
However, the function $\omega_{\lambda, 1}$ 
may not satisfy the constraint $L(\omega_{\lambda, 1}) = 1$, so that we cannot 
consider it as a test function for the maximization 
problem~\eqref{def:max-prob-S}.

\emph{Step 2.} We recover this property by a suitable translation of the restriction to the set 
$S$ of the function $\omega_{\lambda, 1}$. Consider the second order algebraic equation
\begin{equation} 
\label{eq:r_0}
r_0^2 + \bigg( 2 \int_S x_1 \omega_{\lambda, 1}(x) \, dx \bigg) 
r_0 + \int_S |x|^2 \omega_{\lambda, 1}(x) \, dx - 1 = 0.
\end{equation}
Going back to the definition of the function $\omega_{\lambda, 1}$, we observe that 
the support of its restriction to $S$ is included into the ball $\cB( (1, 0), 
(\Lambda + 2 \rho_1) \varepsilon)$. In view of~\eqref{eq:mass-residual} 
and~\eqref{def:rho1}, we deduce that $$\int_S x_1 \omega_{\lambda, 1}(x) \, dx = 
\big( 1 + \cO(\varepsilon) \big) \int_S \omega_{\lambda, 1} = 1 + 
\cO(\varepsilon) ,$$ in the limit $\varepsilon \to 0$. Similarly, we have 
$$\int_S |x|^2 \omega_{\lambda, 1}(x) \, dx = 1 + \cO(\varepsilon) ,$$ as 
$\varepsilon \to 0$. As a consequence, the discriminant $$\Delta = 4 \Big( 
\int_S x_1 \omega_{\lambda, 1}(x) \, dx \Big)^2 - 4 \int_S |x|^2 
\omega_{\lambda, 1}(x) \, dx + 4$$ of~\eqref{eq:r_0} is positive for 
$\varepsilon$ small enough, and we are authorized to set
\begin{equation} 
\label{def:r_0}
r_0 = - \int_S x_1 \omega_{\lambda, 1}(x) \, dx + \bigg( \Big( 
\int_S x_1 \omega_{\lambda, 1}(x) \, dx \Big)^2 - \int_S |x|^2 \omega_{\lambda, 
1}(x) \, dx + 1 \bigg)^\frac{1}{2}.
\end{equation}
Using the inequality 
$|\sqrt{v} - \sqrt{u}| \leq |v - u|$ for $u \geq 1/4$ and $v \geq 1/4$, we can 
estimate the number $r_0$ by
\begin{equation}
\label{eq:ineq-r_0}
|r_0| \leq \big| L(\omega_{\lambda, 1}) - 1 \big|,
\end{equation}
for $\varepsilon$ small 
enough. Invoking once again the definition of the function $\omega_{\lambda, 
1}$, we can combine the property that $1/2 \leq |x| \leq 2$ for $x \in S$ with 
inequality~\eqref{eq:mass-residual} in order to estimate the quantity 
$L(\omega_{\lambda, 1})$ as
\begin{equation}
\label{eq:sur locean des ages} 
\big| L(\omega_{\lambda, 1}) - 1 \big| = \big| L(\omega_{\lambda, 1}) - 
L(\omega_\lambda) \big| \leq 4 \, M(\omega_\lambda, B) \leq C 
\varepsilon^{\gamma_s}.
\end{equation}
Hence we obtain the estimate 
\begin{equation}
\label{eq:estim-r_0}
|r_0| \leq C \varepsilon^{\gamma_s}. 
\end{equation}

At this stage, we set
\begin{equation}
\label{def:omega-lambda-2}
\forall x \in \mfS_N, \quad \omega_{\lambda, 2}(x) = \omega_{\lambda, 1}\big( x - (r_0, 0) 
\big).
\end{equation}
Since the support of the restriction of $\omega_{\lambda, 
1}$ to $S$ is included into the ball $\cB( (1, 0), (\Lambda + 2 \rho_1) 
\varepsilon)$, estimate~\eqref{eq:estim-r_0} is enough to guarantee that the 
function $\omega_{\lambda, 2}$ has a compact support in $S$ when $\varepsilon$ 
is small enough. In particular, we can extend this function to $\RR^2$ by using the $N$-fold 
symmetry, and this extension amounts to translate the restriction of the function $\omega_{\lambda, 1}$ to each angular sector $R_{(2 n \pi)/N} \mfS_N$ by the vector $y_n = R_{(2 n \pi)/N}(r_0, 0)$. In particular, we can check that the function $\omega_{\lambda, 2}$ belongs to the set $X_\lambda$. By construction, it is also angular Steiner symmetric and it satisfies 
the constraint $M(\omega_{\lambda, 2}) = 1$. On the other hand, we compute 
$$L(\omega_{\lambda, 2}) = \int_S |x + (r_0, 0)|^2 \omega_{\lambda, 1}(x) \, dx 
= r_0^2 + 2 r_0 \int_S x_1 \omega_{\lambda, 1}(x) \, dx + \int_S |x|^2 
\omega_{\lambda, 1}(x) \, dx,$$
using the property that $M(\omega_{\lambda, 1}) = 1$. Since the number $r_0$
is a root of~\eqref{eq:r_0}, we infer that $L(\omega_{\lambda, 2}) = 1$.

We finally estimate the energy of the function $\omega_{\lambda, 2}$. Since this 
function is constructed by translating the function $\omega_{\lambda, 1}$ in 
$S$, we first have
\begin{equation}
\label{eq:first-omega_2}
I_s(\omega_{\lambda, 2}, S, S) = I_s(\omega_{\lambda, 1}, S, S).
\end{equation}
Similarly, we also derive from the definition of the function $\omega_{\lambda, 1}$ that
\begin{align*}
I_s(\omega_{\lambda, 2}, & S, \RR^2 \setminus S) - I_s(\omega_{\lambda, 1}, S, \RR^2 \setminus S) \\
& = \sum_{n = 1}^{N - 1} \int_S \int_{R_\frac{2 n \pi}{N} S} \bigg( \Big( \frac{\varepsilon}{|x - x' + y_0 - y_n|}
\Big)^{2 (1 - s)} - \Big( \frac{\varepsilon}{|x - x'|} \Big)^{2 (1 - s)} \bigg)
\omega_{\lambda, 1}(x') \, dx' \omega_{\lambda, 1}(x) \, dx.
\end{align*}
Coming back to~\eqref{eq:large-dist}, we infer that
$$I_s(\omega_{\lambda, 2}, S, \RR^2 \setminus S) - I_s(\omega_{\lambda, 1}, S, \RR^2
\setminus S) \geq - C \varepsilon^{2 (1 - s)} \sum_{n = 1}^{N - 1} \big| y_0 - y_n \big|.$$
Since $|y_n| = r_0$ for any $0 \leq n \leq N - 1$, we deduce from~\eqref{eq:ineq-r_0}
and~\eqref{eq:sur locean des ages} that
$$I_s(\omega_{\lambda, 2}, S, \RR^2 \setminus S) \geq 
I_s(\omega_{\lambda, 1}, S, \RR^2 \setminus S) - C \varepsilon^{2 (1 - s)} 
M(\omega_\lambda, B).$$
Combining this estimate with~\eqref{eq:ne pourrons nous jamais} and~\eqref{eq:first-omega_2} leads to the inequality
\begin{equation}
\label{eq:jeter lancre un seul jour}
E_s(\omega_{\lambda, 2}) - E_s(\omega_\lambda) = \frac{N}{2 
\varepsilon^{2 (1 - s)}} \big( I_s(\omega_{\lambda, 2}, \RR^2, \RR^2) - 
I_s(\omega_\lambda, \RR^2, \RR^2) \big) \geq C \Big( \frac{1}{\varepsilon^{2 (1 
- s)}} - 1 \Big) M(\omega_\lambda, B).
\end{equation}
This inequality contradicts the maximality of the function $\omega_\lambda$ in the limit 
$\varepsilon \to 0$. This concludes the proof of 
Proposition~\ref{prop:condsupp}. \qed

\appendix
\numberwithin{Theo}{section}
\section{The angular Steiner symmetrization}
\label{sec:angular-Steiner}

Our previous construction of co-rotating vortices with $N$-fold symmetry relies several times on a reduction on each fold to angular Steiner symmetric functions. 
This reduction is made possible by the fact that the restriction to each fold of these co-rotating vortices are indeed angular Steiner symmetric. 
In this appendix, we collect the properties of this symmetrization, which we have used for our construction, and we give their proofs. 

To have this appendix self contained, we recall here after the definitions we need.
Consider a non-negative measurable function $\omega$ defined on the angular sector
$$\mfS_N := \Big\{ x=\big( r \cos(\theta), r \sin(\theta) \big)\in\RR^2 : - \frac{\pi}{N} < \theta < \frac{\pi}{N} \Big\}.$$
Its angular Steiner symmetrization $\omega^\sharp$ is defined as the unique even function for the variable $\theta$ such that
$$\omega^\sharp(r, \theta) > \nu \quad \text{if and only if} \quad |\theta| < \frac{1}{2} \meas \Big\{ \theta' \in \Big( - \frac{\pi}{N}, \frac{\pi}{N} \Big) : \omega(r, \theta') > \nu \Big\},$$
for any positive numbers $r$ and $\nu$, and any angle $- \pi/N < \theta < \pi/N$. This function is well-defined and measurable on $\mfS_N$. Moreover, the function $\omega$ is said to be angular Steiner symmetric if and only if $\omega^\sharp = \omega$ almost everywhere.
In view of its definition, the angular Steiner symmetrization satisfies
$$\meas \Big\{ \theta \in \Big( - \frac{\pi}{N}, \frac{\pi}{N} \Big) : \omega^\sharp(r, \theta) > \nu \Big\} = \meas \Big\{ \theta \in \Big( - \frac{\pi}{N}, \frac{\pi}{N} \Big) : \omega(r, \theta) > \nu \Big\},$$
for any positive numbers $r$ and $\nu$, which means that it is a rearrangement with respect to the angular variable $\theta$. As a consequence of the layer-cake representation of non-negative measurable functions, the angular Steiner symmetrization maps the set
$$L_+^p(\mfS_N) := \big\{ \omega \in L^p(\mfS_N) \text{ s.t. } \omega \geq 0 \text{ a.e.} \big\},$$
into itself. This rearrangement preserves several integral quantities as stated in the next lemma.

\begin{Lemma}
\label{lem:inv-C-L}
Let $\omega \in L_{loc,+}^1(\mfS_N)$ and let $f:\RR_+\to\RR_+$ continuous. We have
$$\int_{\mfS_N}f(r)\,\omega(r,\theta)\,r\,dr\,d\theta=\int_{\mfS_N}f(r)\,\omega^\sharp(r,\theta)\,r\,dr\,d\theta,
$$
provided that these quantities are finite.
\end{Lemma}

\begin{proof}
This is a direct consequence of the layer-cake formula
$$\omega(r, \theta) = \int_0^{+ \infty} \mathbbm{1}_{\{ \omega(r,\cdot) > \nu \}}(\theta) \, d\nu,$$
the Fubini theorem and the definition of the angular Steiner symmetrization.
\end{proof}

In contrast, the energy $E$ is not conserved by the angular Steiner symmetrization, but it is increased by this transformation. The proof of this claim relies on the Riesz rearrangement inequality for the angular Steiner symmetrization.

\begin{Lemma}
\label{lem:symmetry}
Let $\omega \in L_+^{\infty}(S)$. We have
$$E(\omega) \leq E(\omega^\sharp).$$
\end{Lemma}

\begin{proof}
Recall first that the quantity $E(\omega)$ is defined by
\begin{equation}
\label{eq:E-N-fold}
E(\omega) = \int_{\frac{1}{2}}^{2} r\, dr \int_{\frac{1}{2}}^{2} r' \, dr' \bigg( \int_{-\frac{\pi}{2 N}}^\frac{\pi}{2 N} \int_{-\frac{\pi}{2 N}}^\frac{\pi}{2 N} \kappa(r, r', \theta - \theta') \, \omega(r, \theta) \, \omega(r', \theta') \, d\theta \, d\theta' \bigg).
\end{equation}
In this expression, the kernel $\kappa$ is equal to
$$\kappa(r, r',\xi) = \sum_{n = 0}^{N - 1} \frac{N c_s}{\big( r^2 + (r')^2 - 2 r r' \cos \big( \xi - \frac{2 \pi n}{N} \big) \big)^{1 - s}},$$
for $\frac{1}{2} < r \neq r' < 2$ and $- \pi/N < \xi < \pi/N$. This function is non-negative and even with respect to the variable $\xi$ due to the identity
$$\cos \Big( - \xi - \frac{2 \pi n}{N} \Big) = \cos \Big( \xi - \frac{2 \pi (N - n)}{N} \Big),$$
which holds for $- \pi/N < \xi < \pi/N$ and $1 \leq n \leq N - 1$.

Moreover, given two fixed numbers $\frac{1}{2} < r \neq r' < 2$, the map $\xi \mapsto \kappa(r, r', \xi)$ is smooth on $(- \pi/N, \pi/N)$, and its derivative is given by
$$\frac{\partial \kappa}{\partial \xi}(r, r', \xi) = - 2 c_s (1 - s) r r' \sum_{n = 0}^{N - 1} \frac{\sin \big( \xi - \frac{2 \pi n}{N} \big)}{\big( r^2 + (r')^2 - 2 r r' \cos \big( \xi - \frac{2 \pi n}{N} \big) \big)^{2 - s}}.$$
We claim that this quantity is negative on $(0, \pi/N)$, while it is positive on $(- \pi/N, 0)$. More precisely, we have

\begin{Lemma}
\label{lem:combinatorics}
Let $N \geq 2$, $0 \leq s \leq 1$ and $0 < \varrho < 1$. Consider the function $\phi : [- 1, 1] \to \RR$ defined by
\begin{equation}
\label{def:phi}
\phi(x) := \frac{1}{(1 - \varrho x)^{2 - s}},
\end{equation}
for any number $- 1 \leq x \leq 1$, and set
$$\Phi(\xi) = \sum_{n = 0}^{N - 1} \sin \Big( \xi - \frac{2 \pi n}{N} \Big) \phi \Big( \cos \Big( \xi - \frac{2 \pi n}{N} \Big) \Big),$$
for any number $\xi \in \RR$. There exists a continuous, positive and $2 \pi/N$-periodic function $F : \RR \to \RR$ such that
\begin{equation}
\label{decromieres}
\Phi(\xi) = \sin(N \xi) \, F(\xi),
\end{equation}
for any number $\xi \in \RR$.
\end{Lemma}

Our previous claim directly follows from applying formula~\eqref{decromieres} with $\varrho = 2 r r'/(r^2 + (r')^2) \in (0, 1)$. For sake of clarity, we postpone the proof of Lemma~\ref{lem:combinatorics} and now complete the proof of Lemma~\ref{lem:symmetry}.

Coming back to~\eqref{eq:E-N-fold}, we fix two positive numbers $r \neq r'$ and extend the functions $\theta \mapsto \omega(r, \theta)$ and $\theta \mapsto \omega(r', \theta)$ to $\RR$ by letting them equal to $0$ outside the interval $(- \pi/2N, \pi/2N)$. Similarly, we extend the map $\theta \mapsto \kappa(r, r', \theta)$ to $\RR$ by letting it equal to $0$ outside the interval $(- \pi/N, \pi/N)$. We then have
$$\int_{-\frac{\pi}{2 N}}^\frac{\pi}{2 N} \int_{-\frac{\pi}{2 N}}^\frac{\pi}{2 N} \kappa(r, r', \theta - \theta') \, \omega(r, \theta) \, \omega(r', \theta') \, d\theta \, d\theta' = \int_\RR \int_\RR \kappa(r, r', \theta - \theta') \, \omega(r, \theta) \, \omega(r', \theta') \, d\theta \, d\theta',$$
and we can apply the Riesz rearrangement inequality (see e.g.~\cite[Lemma 3.6]{LiebLos0}) in order to obtain
$$\int_{-\frac{\pi}{2 N}}^\frac{\pi}{2 N} \int_{-\frac{\pi}{2 N}}^\frac{\pi}{2 N} \kappa(r, r', \theta - \theta') \, \omega(r, \theta) \, \omega(r', \theta') \, d\theta \, d\theta' \leq \int_\RR \int_\RR \kappa^\sharp(r, r', \theta - \theta') \, \omega^\sharp(r, \theta) \, \omega^\sharp(r', \theta') \, d\theta \, d\theta'.$$
Since the map $\xi \mapsto \kappa(r, r', \xi)$ is even and non-increasing on $\RR_+$, we have
$$\kappa^\sharp(r, r', \theta - \theta') = \kappa(r, r', \theta - \theta'),$$
for any $(\theta, \theta') \in \RR^2$. Moreover, since the functions $\theta \mapsto \omega(r, \theta)$ and $\theta \mapsto \omega(r', \theta)$ are supported in $(- \pi/2N, \pi/2N)$, so are the functions $\theta \mapsto \omega^\sharp(r, \theta)$ and $\theta \mapsto \omega^\sharp(r', \theta)$, and this property leads to
$$\int_{-\frac{\pi}{2 N}}^\frac{\pi}{2 N} \int_{-\frac{\pi}{2 N}}^\frac{\pi}{2 N} \kappa(r, r', \theta - \theta') \, \omega(r, \theta) \, \omega(r', \theta') \, d\theta \, d\theta' \leq \int_{-\frac{\pi}{2 N}}^\frac{\pi}{2 N} \int_{-\frac{\pi}{2 N}}^\frac{\pi}{2 N} \kappa(r, r', \theta - \theta') \, \omega^\sharp(r, \theta) \, \omega^\sharp(r', \theta') \, d\theta \, d\theta'.$$
Lemma~\ref{lem:symmetry} finally follows from introducing this inequality into~\eqref{eq:E-N-fold}.
\end{proof}

We now provide the proof of Lemma~\ref{lem:combinatorics}.

\begin{proof}[Proof of Lemma~\ref{lem:combinatorics}]
For sake of simplicity, we set
$$\xi_n = \xi - \frac{2 \pi n}{N},$$
for any number $\xi \in \RR$ and any integer $0 \leq n \leq N - 1$. With this notation at hand, the function $\Phi$ rewrites as
$$\Phi(\xi) = \sum_{n = 0}^{N - 1} \sin(\xi_n) \phi \big( \cos(\xi_n) \big),$$
and we can express it as in~\eqref{decromieres} by developing the following inductive argument.

The first step, as the subsequent ones, relies on a Taylor expansion of the function $\phi$. We define the maps
$$T_p[f](x) := \int_0^1 f^{(p)}(s x) (1 - s)^{p - 1} \, ds$$
for any integer $p \geq 1$, any number $-1 \leq x \leq 1$ and any function $f \in \cC^\infty([- 1, 1])$. These maps allow to write the remainder term in the Taylor formula as
$$f(x) = \sum_{k = 0}^{p - 1} \frac{f^{(k)}(0)}{k !} x^k + \frac{x^p}{(p - 1)!} T_p[f](x).$$
Since the function $\phi$ is smooth on $\RR$, we infer from this formula that
\begin{equation}
\label{ASM}
\begin{split}
\Phi(\xi) = \sum_{k = 0}^{N - 2} \frac{\phi^{(k)}(0)}{k !} & \sum_{n = 0}^{N - 1} \sin(\xi_n) \cos(\xi_n)^k\\
& + \frac{1}{(N - 2)!} \sum_{n = 0}^{N - 1} \sin(\xi_n) \cos(\xi_n)^{N - 1} T_{N - 1}[\phi](\cos(\xi_n)).
\end{split}
\end{equation}

We next claim that the trigonometric sums $\sum_{n = 0}^{N - 1} \sin(\xi_n) \cos(\xi_n)^k$ are equal to $0$ when $0 \leq k \leq N - 2$. For further use, we more generally set
\begin{equation}
\label{def:skl}
\sigma_{k, \ell}(\xi) := \sum_{n = 0}^{N - 1} \sin(\ell \xi_n) \cos(\xi_n)^k,
\end{equation}
for any integer $(k, \ell) \in \NN^2$, and we check that
\begin{equation}
\label{tadjer}
\sigma_{k, \ell}(\xi) = 0,
\end{equation}
provided that $0 \leq k + \ell \leq N - 1$. Indeed, the Euler formula for the cosine function and the binomial theorem give
$$\sigma_{k, \ell}(\xi) = \frac{1}{2^k} \sum_{n = 0}^{N - 1} \sin(\ell \xi_n) \sum _{j = 0}^k \binom{k}{j} e^{i (k - 2 j) \xi_n}.$$
Since the sum $\sigma_{k, \ell}(\xi)$ is real-valued, this expression reduces to
$$\sigma_{k, \ell}(\xi) = \frac{1}{2^k} \sum _{j = 0}^k \binom{k}{j} \sum_{n = 0}^{N - 1} \sin(\ell \xi_n) \cos \big( (k - 2 j) \xi_n \big),$$
so that
\begin{equation}
\label{falgoux}
\sigma_{k, \ell}(\xi) = \frac{1}{2^{k + 1}} \sum _{j = 0}^k \binom{k}{j} \sum_{n = 0}^{N - 1} \Big( \sin \big( (k + \ell - 2 j) \xi_n \big) + \sin \big( (2 j + \ell - k) \xi_n \big) \Big).
\end{equation}
We rewrite this formula as
\begin{equation}
\label{uhila}
\sigma_{k, \ell}(\xi) = \frac{1}{2^{k + 1}} \sum _{j = 0}^k \binom{k}{j} \Im \bigg( e^{i (k + \ell - 2 j) \xi} \sum_{n = 0}^{N - 1} \Big( e^\frac{2 \pi i (2 j - k - \ell)}{N} \Big)^n+ e^{i (2 j + \ell - k) \xi} \sum_{n = 0}^{N - 1} \Big( e^\frac{2 \pi i (k - \ell - 2 j)}{N} \Big)^n \bigg).
\end{equation}
Recall that
$$\sum_{n = 0}^{N - 1} \Big( e^\frac{2 p \pi i}{N} \Big)^n = \begin{cases} N & \text{if } p \text{ is a multiple of }N,\\ 0 & \text{elsewhere}, \end{cases}$$
and observe that the integers $2 j - \ell - k$ and $k - \ell - 2 j$ lie in the interval $[- (N - 1), N - 1]$ due to the constraints $0 \leq k + \ell \leq N - 1$ and $0 \leq j \leq k$. As a consequence, the sums with respect to $n$ in~\eqref{uhila} are equal to $0$ except if $2 j - \ell -k = 0$ for the first sum, respectively $k - \ell- 2 j = 0$ for the second one. In these two cases, the quantity in the imaginary part is real-valued, so that the quantity $\sigma_{k, \ell}(\xi)$ is indeed equal to $0$.

Invoking~\eqref{tadjer} with $0 \leq k \leq N - 2$ and $\ell = 1$, we can simplify~\eqref{ASM} as
\begin{equation}
\label{azema}
\Phi(\xi) = \frac{1}{(N - 2)!} \sum_{n = 0}^{N - 1} \sin(\xi_n) \cos(\xi_n)^{N - 1} T_{N - 1}[\phi](\cos(\xi_n)),
\end{equation}
and put the focus on the trigonometric polynomial $\xi \mapsto \sin(\xi) \cos(\xi)^{N - 1}$. As before, we more generally consider the functions $\xi \mapsto \sin(\ell \xi) \cos(\xi)^{N - \ell}$ for an integer $1 \leq \ell \leq N - 1$. Arguing as for the proof of~\eqref{falgoux} and using the identity $\sin(N \xi_n) = \sin(N \xi)$ for $0 \leq n \leq N - 1$, we observe that
\begin{equation}
\label{slimani}
\sin(\ell \xi_n) \cos(\xi_n)^{N - \ell} = \frac{1}{2^{N - \ell}} \bigg( \sin(N \xi) + \sum_{k = 1}^{N - \ell} \binom{N - \ell}{k} \sin \big( (N - 2 k) \xi_n \big) \bigg).
\end{equation}
In this expression, the quantities $\sin \big( (N - 2 k) \xi_n \big)$ can appear only once, with a positive multiplicative factor, or twice, with opposite multiplicative factors. For further use, we aim at eliminating this late behavior. For $2 \ell \geq N$, this behavior is excluded, so that we now assume that $2 \ell < N$. In this case, we use the change of indices $j = N - k$ for $N/2 < k \leq N - \ell$ in order to decompose the sum over $k$ in~\eqref{slimani} as
\begin{equation}
\label{zirakashvili}
\begin{split}
\sum_{k = 1}^{N - \ell} \binom{N - \ell}{k} \sin \big( (N - 2 k) \xi_n \big) = \sum_{1 \leq k < \ell} & \binom{N - \ell}{k} \sin \big( (N - 2 k) \xi_n \big)\\
& + \sum_{\ell \leq k < \frac{N}{2}} \Big( \binom{N - \ell}{k} - \binom{N - \ell}{N - k} \Big) \sin \big( (N - 2 k) \xi_n \big).
\end{split}
\end{equation}
Note here that
\begin{equation}
\label{beheregaray}
\binom{N - \ell}{k} - \binom{N - \ell}{N - k} = \frac{(N - \ell)!}{k! (N - k)!} \Big( \prod_{j = 0}^{\ell - 1} (N - k - j) - \prod_{j = 0}^{\ell - 1} (k - j) \Big) > 0,
\end{equation}
since $N - k > k$ when $\ell \leq k < N/2$. Hence, rewriting~\eqref{slimani} as in~\eqref{zirakashvili} provides an expression in which each quantity $\sin \big( (N - 2 k) \xi_n \big)$ appears only once, with a positive multiplicative factor.

Going back to~\eqref{azema}, we can combine~\eqref{slimani} and~\eqref{zirakashvili} for $\ell = 1$ in order to obtain 
\begin{equation}
\label{goutta}
\begin{split}
\Phi(\xi) = & \frac{\sin(N \xi)}{ 2^{N - 1} (N - 2)!} \sum_{n = 0}^{N - 1} T_{N - 1}[\phi](\cos(\xi_n))\\
& + \frac{1}{2^{N - 1} (N - 2)!} \sum _{1 \leq k < \frac{N}{2}} \bigg( \binom{N - 1}{k} - \binom{N - 1}{N - k} \bigg) \sum_{n = 0}^{N - 1} \sin \big( (N - 2k) \xi_n \big) T_{N - 1}[\phi](\cos(\xi_n)).
\end{split}
\end{equation}
At this stage, we set
\begin{equation}
\label{def:F1}
F^1(\xi) := \frac{1}{ 2^{N - 1} (N - 2)!} \sum_{n = 0}^{N - 1} T_{N - 1}[\phi](\cos(\xi_n)),
\end{equation}
and
\begin{equation}
\label{def:R1}
R_{N - 2 k}^1(x) = \frac{1}{2^{N - 1} (N - 2)!} \bigg( \binom{N - 1}{k} - \binom{N - 1}{N - k} \bigg) T_{N - 1}[\phi](x),
\end{equation}
for $1 \leq k < N/2$ and $- 1 \leq x \leq 1$. With this notation at hand, we can write~\eqref{goutta} as
\begin{equation}
\label{bes}
\Phi(\xi) = \sin(N \xi) F^1(\xi) + \sum _{1 \leq k < \frac{N}{2}} \sum_{n = 0}^{N - 1} \sin \big( (N - 2 k) \xi_n \big) R_{N - 2 k}^1(\cos(\xi_n)),
\end{equation}
and we observe that the sums
$$\sum_{n = 0}^{N - 1} \sin \big( (N - 2k) \xi_n \big) R_{N - 2 k}^1(\cos(\xi_n))$$
have exactly the same form as the map $\Phi$ with the function $\phi$ being replaced by the functions $R_{N - 2 k}^1$. As a consequence, we can expect that an inductive argument eventually provides an expression of the map $\Phi$ as in~\eqref{decromieres}.

Before going into this inductive argument, we now check that the function $F^1$ is positive, while the remainder terms $R_{N - 2 k}^1$ are positive and absolutely monotone, which means that all their derivatives are positive. This claim originates into the two following properties.

First, in view of~\eqref{def:phi}, the function $\phi$ is well-defined and smooth on $[- 1, 1] \subset (- 1/\varrho, 1/\varrho)$, and its successive derivatives are given by
$$\phi^{(p)}(x) = \frac{(2 - s) (3 - s) \ldots (p + 1 - s) \varrho^p}{(1 - \varrho x)^{p + 2 - s}} >0 ,$$
for any $p \geq 1$ and any $x \in [- 1, 1]$. Therefore, the function $\phi$ is absolutely monotone. Actually, it extends to an analytic function on the interval $(- 1/\varrho, 1/\varrho)$ since it can be expanded as the power series
Let $R > 0$ We next assume for the sake of a contradiction the existence of a point $x \in \cB( (1,0), 2 R \varepsilon)^c$ such that $\psi_\lambda(x) \geq 0$. When $\varepsilon$ is small enough, we can invoke Lemma~\ref{lem:decay} in order to find $y_1 \in S$ such that $\psi_\lambda(y_1) < 0$ and the two following conditions are satisfied
\begin{equation}
\label{jedrasiak}
\phi(x) = \sum_{p = 0}^{+ \infty} \alpha_{p, s} \rho^p x^p,
\end{equation}
for $ x \in (- 1/\varrho, 1/\varrho)$. In this formula, the coefficients $\alpha_{p, s}$ are given by
\begin{equation}
\label{timani}
\alpha_{0, s} := 1 \quad \text{and} \quad \alpha_{p, s} := \frac{(2 - s)(3 - s) \ldots (p + 1 - s)}{p !} \text{ for } p \geq 1,
\end{equation}
and all of them are positive.

Second, we remark that, if a function $f \in \cC^\infty([- 1, 1], \RR)$ is positive and absolutely monotone, so are the maps $T_p[f]$ for any $p \geq 1$. This is a direct consequence of the computation
$$T_p[f]^{(q)}(x) = \int_0^1 f^{(p + q)}(s x) s^q (1 - s)^{p - 1} \, ds > 0.$$
Hence, we can invoke~\eqref{beheregaray},~\eqref{def:F1} and~\eqref{def:R1} to conclude that the function $F^1$ is positive, while the remainder terms $R_{N - 2 k}^1$ are positive and absolutely monotone.

We now go inside the inductive argument. Given an integer $p \geq 1$, we assume that we have constructed positive and smooth functions $(F^q)_{1 \leq q \leq p}$ on $\RR$, and positive, smooth and absolutely monotone functions $(R_{N - 2 k}^q)_{1 \leq q \leq p, 1 \leq k < N/2}$ on $[- 1, 1]$, such that
\begin{equation}
\label{sadourny}
\Phi(\xi) = \sin(N \xi) \sum_{q = 1}^p F^q(\xi) + \sum _{1 \leq k < \frac{N}{2}} \sum_{n = 0}^{N - 1} \sin \big( (N - 2k) \xi_n \big) R_{N - 2 k}^p(\cos(\xi_n)),
\end{equation}
for any $\xi \in \RR$. In order to construct the function $F^{p + 1}$ and the remainders $R_{N - 2 k}^{p + 1}$, we apply the same strategy as in the first inductive argument above. We fix an integer $1 \leq k < N/2$ and we invoke the Taylor formula so as to obtain
\begin{align*}
\sum_{n = 0}^{N - 1} \sin \big( (N - 2 k) \xi_n \big) & R_{N - 2 k}^p(\cos(\xi_n)) = \sum_{j = 0}^{2 k - 1} \frac{\big( R_{N - 2 k}^p \big)^{(j)}(0)}{j !} \sum_{n = 0}^{N - 1} \sin \big( (N - 2k) \xi_n \big) \cos(\xi_n)^j\\
& + \frac{1}{(2 k - 1)!} \sum_{n = 0}^{N - 1} \sin \big( (N - 2k) \xi_n \big) \cos(\xi_n)^{2 k} T_{2 k}[R_{N - 2 k}^p](\cos(\xi_n)).
\end{align*}
In view of~\eqref{tadjer}, the double sum in the right-hand side of this identity is equal to $0$, so that
\begin{align*}
\sum_{n = 0}^{N - 1} \sin \big( (N - 2 k) \xi_n \big) & R_{N - 2 k}^p(\cos(\xi_n)) =\\
& \frac{1}{(2 k - 1)!} \sum_{n = 0}^{N - 1} \sin \big( (N - 2k) \xi_n \big) \cos(\xi_n)^{2 k} T_{2 k}[R_{N - 2 k}^p](\cos(\xi_n)).
\end{align*}
When $k \leq N/4$, we next use~\eqref{slimani} in order to get
\begin{equation}
\label{iturria}
\begin{split}
\sum_{n = 0}^{N - 1} \sin \big( (N - 2 k) \xi_n \big) & R_{N - 2 k}^p(\cos(\xi_n)) = \frac{\sin(N \xi)}{4^k (2 k - 1)!} \sum_{n = 0}^{N - 1} T_{2 k}[R_{N - 2 k}^p](\cos(\xi_n))\\
& + \frac{1}{4^k (2 k - 1)!} \sum_{j = 1}^{2 k} \binom{2 k}{j} \sum_{n = 0}^{N - 1} \sin \big( (N - 2 j) \xi_n \big) T_{2 k}[R_{N - 2 k}^p](\cos(\xi_n)).
\end{split}
\end{equation}
For $k > N/4$, we similarly invoke~\eqref{zirakashvili} so as to obtain
\begin{align*}
\sum_{n = 0}^{N - 1} \sin & \big( (N - 2 k) \xi_n \big) R_{N - 2 k}^p(\cos(\xi_n)) = \frac{\sin(N \xi)}{4^k (2 k - 1)!} \sum_{n = 0}^{N - 1} T_{2 k}[R_{N - 2 k}^p](\cos(\xi_n))\\
& + \frac{1}{4^k (2 k - 1)!} \sum_{1 \leq j < N - 2 k} \binom{2 k}{j} \sum_{n = 0}^{N - 1} \sin \big( (N - 2 j) \xi_n \big) T_{2 k}[R_{N - 2 k}^p](\cos(\xi_n))\\
& + \frac{1}{4^k (2 k - 1)!} \sum_{N - 2 k \leq j < \frac{N}{2}} \bigg( \binom{2 k}{j} - \binom{2 k}{N - j} \bigg) \sum_{n = 0}^{N - 1} \sin \big( (N - 2 j) \xi_n \big) T_{2 k}[R_{N - 2 k}^p](\cos(\xi_n)).
\end{align*}
Setting
\begin{equation}
\label{def:Fp+1}
F^{p + 1}(\xi) := \sum_{1 \leq k < \frac{N}{2}} \frac{1}{4^k (2 k - 1)!} \sum_{n = 0}^{N - 1} T_{2 k}[R_{N - 2 k}^p](\cos(\xi_n)),
\end{equation}
as well as
\begin{equation}
\label{def:Rp+1}
\begin{split}
R_{N - 2 k}^{p + 1}(x) = \sum_{\max \big\{ 1, \frac{k}{2} \big\} \leq j < \frac{N - k}{2}} \binom{2 j}{k} & \frac{ T_{2 j}[R_{N - 2 j}^p](x)}{4^j (2 j - 1)!}\\
& + \sum_{\frac{N - k}{2} \leq j < \frac{N}{2}} \bigg( \binom{2 j}{k} - \binom{2 j}{N - k} \bigg) \frac{ T_{2 j}[R_{N - 2 j}^p](x)}{4^j (2 j - 1)!},
\end{split}
\end{equation}
we conclude that
$$\Phi(\xi) = \sin(N \xi) \sum_{q = 1}^{p + 1} F^q(\xi) + \sum _{1 \leq k < \frac{N}{2}} \sum_{n = 0}^{N - 1} \sin \big( (N - 2k) \xi_n \big) R_{N - 2 k}^{p + 1} (\cos(\xi_n)).$$
Moreover, we can prove as for the functions $F^1$ and $R_{N - 2 k}^1$ that the map $F^{p + 1}$ is positive and smooth on $\RR$, while the maps $R_{N - 2 k}^{p + 1}$ are positive, smooth and absolutely monotone on $[-1, 1]$. This completes the inductive argument and establishes the validity of~\eqref{sadourny} for any $p \geq 1$.

We finally conclude the proof of Lemma~\ref{lem:combinatorics} by showing that
\begin{equation}
\label{vahaamahina}
\sum_{1 \leq k < N/2} \big| R_{N - 2 k}^p(x) \big| \to 0,
\end{equation}
as $p \to + \infty$, uniformly with respect to $x \in [- 1, 1]$. In order to establish this convergence, we use the analytic expansion of the function $\phi$ in~\eqref{jedrasiak}. More generally, we consider an analytic function $f$ on the interval $(- 1/\varrho, 1/\varrho)$, which we expand as the power series
$$f(x) = \sum_{k = 0}^{+ \infty} a_k x^k,$$
for $ x \in (- 1/\varrho, 1/\varrho)$. Given an integer $m \geq 1$ and a number $x \in [- 1, 1]$, we compute
$$T_m[f](x) = \sum_{k = 0}^{+ \infty} a_{k + m} (k + m) \ldots (k + 1) x^k \int_0^1 s^k (1 - s)^{m - 1} \, ds.$$
Integrating by parts and arguing by induction provide
$$\int_0^1 s^k (1 - s)^{m - 1} \, ds = \frac{m - 1}{k + 1} \int_0^1 s^{k + 1} (1 - s)^{m - 2} \, ds = \frac{(m - 1)!}{(k + m) \ldots (k + 1)}.$$
Hence we obtain
$$T_m[f](x) = (m - 1)! \, \tau_m f(x),$$
where the notation $\tau_m f$ refers to the analytic function $f$ given by the power series
$$\tau_m f(x) = \sum_{k = 0}^{+ \infty} a_{k + m} x^k,$$
for any number $x \in (- 1/\varrho, 1/\varrho)$.

With this expression at hand, we rewrite~\eqref{def:Rp+1} as
$$R_{N - 2 k}^{p + 1}(x) = \sum_{\frac{k}{2} \leq j < \frac{N}{2}} b_k^{2 j} \, \tau_{2 j} R_{N - 2 j}^p(x),$$
with
$$b_k^j = \begin{cases} \binom{j}{k} \frac{1}{2^j} & \text{if } k \leq j < N - k,\\[5pt] \Big( \binom{j}{k} - \binom{j}{N - k} \Big) \frac{1}{2^j} & \text{if } N - k \leq j < N. \end{cases}$$
Since
$$R_{N - 2 k}^1(x) = b_k^{N - 1} \, \tau_{N - 1} \phi(x),$$
by~\eqref{def:R1}, a direct inductive argument provides
\begin{equation}
\label{lapandry}
\begin{split}
R_{N - 2 k}^{p + 1}(x) = \sum_{\max \big\{ 1, \frac{k}{2} \big\} \leq j_1 < \frac{N}{2}} b_k^{2 j_1} & \sum_{\max \big\{ 1, \frac{j_1}{2} \big\} \leq j_2 < \frac{N}{2}} b_{j_1}^{2 j_2} \ldots\\
& \ldots \sum_{\max \big\{ 1, \frac{j_{p - 1}}{2} \big\} \leq j_p < \frac{N}{2}} \, b_{j_{p - 1}}^{2 j_p}\, b_{j_p}^{N - 1} \tau_{2 j_1 + 2 j_2 \ldots 2 j_p + N - 1} \phi(x).
\end{split}
\end{equation}
In view of~\eqref{jedrasiak}, we know that
$$\tau_m \phi(x) = \sum_{k = 0}^{+ \infty} \alpha_{k + m, s} \rho^{k + m} x^k.$$
Moreover, the radius of convergence of the power series $\sum_{k \geq 0} \alpha_{k, s} x^k$ is equal to $1$. Therefore, given any number $0 < \sigma < 1$, there exists a positive number $M_\sigma$ such that
$$\alpha_{k , s} \sigma^k \leq M_\sigma,$$
for any integer $k \in \NN$. When $\tau < \sigma < 1$ and $- 1 \leq x \leq 1$, we are led to the estimate
$$\big| \tau_m \phi(x) \big| \leq M_\sigma \frac{\tau^m}{\sigma^m} \sum_{k = 0}^{+ \infty} \frac{\rho^k}{\sigma^k} = \frac{M_\sigma \rho^m}{\sigma^{m - 1} (\sigma - \rho)}.$$
Inserting this inequality into~\eqref{lapandry} gives
\begin{align*}
\big| R_{N - 2 k}^{p + 1}(x) \big| \leq \frac{M_\sigma \sigma}{\sigma - \rho} & \sum_{\max \big\{ 1, \frac{k}{2} \big\} \leq j_1 < \frac{N}{2}} b_k^{2 j_1} \sum_{\max \big\{ 1, \frac{j_1}{2} \big\} \leq j_2 < \frac{N}{2}} b_{j_1}^{2 j_2} \ldots\\
& \ldots \sum_{\max \big\{ 1, \frac{j_{p - 1}}{2} \big\} \leq j_p < \frac{N}{2}} \, b_{j_{p - 1}}^{2 j_p}\, b_{j_p}^{N - 1} \bigg( \frac{\rho}{\sigma} \bigg)^{2 j_1 + 2 j_2 \ldots 2 j_p + N - 1}.
\end{align*}
Hence, we deduce from the inequality $\rho < \sigma$ that
\begin{equation}
\label{yato}
\begin{split}
\big| R_{N - 2 k}^{p + 1}(x) \big| & \leq \frac{M_\sigma \sigma}{\sigma - \rho} \bigg( \frac{\rho}{\sigma} \bigg)^{2 p + N - 1} \times\\
& \times \sum_{\max \big\{ 1, \frac{k}{2} \big\} \leq j_1 < \frac{N}{2}} b_k^{2 j_1} \sum_{\max \big\{ 1, \frac{j_1}{2} \big\} \leq j_2 < \frac{N}{2}} b_{j_1}^{2 j_2} \ldots \sum_{\max \big\{ 1, \frac{j_{p - 1}}{2} \big\} \leq j_p < \frac{N}{2}} \, b_{j_{p - 1}}^{2 j_p}\, b_{j_p}^{N - 1}.
\end{split}
\end{equation}
At this stage, we deduce from the discrete Fubini theorem
\begin{align*}
\sum_{1 \leq k \leq \frac{N}{2}} \, \sum_{\max \big\{ 1, \frac{k}{2} \big\} \leq j < \frac{N}{2}} b_k^{2 j} \alpha_j & \leq \sum_{1 \leq k \leq \frac{N}{2}} \, \sum_{\max \big\{ 1, \frac{k}{2} \big\} \leq j < \frac{N}{2}} \binom{2 j}{k} \frac{\alpha_j}{2^{2 j}}\\
& \leq \sum_{1 \leq j \leq \frac{N}{2}} \, \sum_{1 \leq k \leq 2 j} \binom{2 j}{k} \frac{\alpha_j}{2^{2 j}} = \sum_{1 \leq j \leq \frac{N}{2}} \alpha_j,
\end{align*}
for any non-negative numbers $\alpha_1$, $\alpha_2$, $\ldots$ and $\alpha_j$. In view of~\eqref{yato}, this yields
$$\sum_{1 \leq k < \frac{N}{2}} \big| R_{N - 2 k}^{p + 1}(x) \big| \leq \frac{M_\sigma \sigma}{\sigma - \rho} \bigg( \frac{\rho}{\sigma} \bigg)^{2 p + N - 1} \sum_{1 \leq j < \frac{N}{2}} b_j^{N - 1}.$$
Since
$$\sum_{1 \leq j < \frac{N}{2}} b_j^{N - 1} \leq \frac{1}{2^{N - 1}} \sum_{j = 1}^{N - 1} \binom{N - 1}{j} \leq 1,$$
we conclude that
$$\sum_{1 \leq k < \frac{N}{2}} \big| R_{N - 2 k}^{p + 1}(x) \big| \leq \frac{M_\sigma \sigma}{\sigma - \rho} \bigg( \frac{\rho}{\sigma} \bigg)^{2 p + N - 1},$$
which is enough to obtain the uniform convergence in~\eqref{vahaamahina}.

Applying this convergence to~\eqref{sadourny} leads to the identity
$$\Phi(\xi) = \sin(N \xi) \, F(\xi),$$
for any number $\xi \in \RR$. The map $F$ in this formula is defined as the function series
$$F(\xi) = \sum_{q = 1}^{+ \infty} F^q(\xi).$$
Since the convergence of this series is uniform on $\RR$, and all the functions $F^q$ are continuous and positive, so is the map $F$. Its $2 \pi/N$-periodicity then follows from its continuity and the $2 \pi/N$-periodicity of the functions $\Phi$ and $\xi \mapsto \sin(N \xi)$. This completes the proof of Lemma~\ref{lem:combinatorics}.
\end{proof}

\begin{merci}
The authors acknowledge support from the project ``Dispersive and random waves'' (ANR-18-CE40-0020-01) of the Agence Nationale de la Recherche.
\end{merci}

\bibliographystyle{plain}
\bibliography{bibliography}

\begin{thebibliography}{10}

\bibitem{AoDavDPMW}
W.~Ao, J.~D\'avila, M.~Del Pino, M.~Musso, and J.~Wei.
\newblock Travelling and rotating solutions to the generalized inviscid surface
  quasi-geostrophic equation.
\newblock {\em Preprint}, 2020.
\newblock {\tt https://arxiv.org/pdf/2008.12911.pdf}.

\bibitem{Arnold}
V.I. Arnold.
\newblock {Mathematical methods of classical mechanics}.
\newblock {\em Springer Verlag}, 1978.

\bibitem{BergerFraenkel}
M.S. Berger and L.E. Fraenkel.
\newblock A global theory of steady vortex rings in an ideal fluid.
\newblock {\em Acta Math.}, 132(1):13--51, 1974.

\bibitem{BucShVi1}
T.~Buckmaster, S.~Shkoller, and V.~Vicol.
\newblock Nonuniqueness of weak solutions to the {SQG} equations.
\newblock {\em Comm. Pure Appl. Math.}, 72(9):1809--1874, 2019.

\bibitem{Burton1}
G.R. Burton.
\newblock Steady symmetric vortex pairs and rearrangements.
\newblock {\em Proc. Roy. Soc. Edinburgh A}, 108(3-4):269--290, 1988.

\bibitem{CasCoGS1}
A.~Castro, D.~C\'ordoba, and J.~G\'omez-Serrano.
\newblock Global smooth solutions for the inviscid {SQG} equation.
\newblock {\em Mem. Amer. Math. Soc.}, 266(1292), 2020.
\newblock 89 pp.

\bibitem{ConMaTa1}
P.~Constantin, A.J. Majda, and E.~Tabak.
\newblock Formation of strong fronts in the {$2$}-{D} quasigeostrophic thermal
  active scalar.
\newblock {\em Nonlinearity}, 7(6):1495--1533, 1994.

\bibitem{GodaCad0}
L.~Godard-Cadillac.
\newblock {\em Les vortex quasi-g\'eostrophiques et leur d\'esingularisation}.
\newblock PhD thesis, Sorbonne Université, 2020.

\bibitem{GravejatSmets}
P.~Gravejat and D.~Smets.
\newblock Smooth travelling-wave solutions to the inviscid surface
  quasi-geostrophic equation.
\newblock {\em Int. Math. Res. Not.}, IMRN 2019(6):1744–1757, 2019.

\bibitem{HassHmid}
Z.~Hassainia and T.~Hmidi.
\newblock On the v-states for the generalized quasi-geostrophic equations.
\newblock {\em Comm. Math. Phys.}, 337(1):321–377, 2015.

\bibitem{HePiGaS1}
I.M. Held, R.T. Pierrehumbert, S.T. Garner, and K.L. Swanson.
\newblock Surface quasi-geostrophic dynamics.
\newblock {\em J. Fluid Mech.}, 282:1--20, 1995.

\bibitem{HmidMat1}
T.~Hmidi and J.~Mateu.
\newblock Existence of corotating and counter-rotating vortex pairs for active
  scalar equations.
\newblock {\em Comm. Math. Phys.}, 350(2):699--747, 2017.

\bibitem{KiseNaz1}
A.~Kiselev and F.~Nazarov.
\newblock A simple energy pump for the surface quasi-geostrophic equation.
\newblock In H.~Holden and K.H. Karlsen, editors, {\em Nonlinear partial
  differential equations}, volume~7 of {\em Abel Symposia}, pages 175--179.
  Springer-Verlag, Berlin, Heidelberg, 2012.

\bibitem{LiebLos0}
E.H. Lieb and M.~Loss.
\newblock {\em Analysis}, volume~14 of {\em Graduate Studies in Mathematics}.
\newblock Amer. Math. Soc., Providence, {Second} edition, 2001.

\bibitem{MarcPul0}
C.~Marchioro and M.~Pulvirenti.
\newblock {\em Mathematical theory of incompressible nonviscous fluids},
  volume~96 of {\em Applied mathematical sciences}.
\newblock Springer-Verlag, New York, 1994.

\bibitem{Norbury}
J.~Norbury.
\newblock Steady planar vortex pairs in an ideal fluid.
\newblock {\em Comm. Pure Appl. Math.}, 28(6):679--700, 1975.

\bibitem{Resnick0}
S.G. Resnick.
\newblock {\em Dynamical problems in non-linear advective partial differential
  equations}.
\newblock PhD thesis, University of Chicago, 1995.

\bibitem{Riesz1}
F.~Riesz.
\newblock {Sur une in{\'e}galit{\'e} int{\'e}grale}.
\newblock {\em Jour. of the Lond. Math. Soc.}, 5:162--168, 1930.

\bibitem{SmetVSc1}
D.~Smets and J.~Van Schaftingen.
\newblock Desingularization of vortices for the {Euler} equation.
\newblock {\em Arch. Ration. Mech. Anal.}, 198(3):869--925, 2010.

\bibitem{Turking1}
B.~Turkington.
\newblock On steady vortex flow in two dimensions. {I}.
\newblock {\em Comm. Partial Differential Equations}, 8(9):999--1030, 1983.

\bibitem{Turking2}
B.~Turkington.
\newblock On steady vortex flow in two dimensions. {II}.
\newblock {\em Comm. Partial Differential Equations}, 8(9):1031--1071, 1983.

\bibitem{Turking3}
B.~Turkington.
\newblock Corotating steady vortex flows with {N}-fold symmetry.
\newblock {\em Nonlinear Anal.}, 9(4):351--369, 1985.

\end{thebibliography}

\end{document}